\documentclass[11pt]{article}

\usepackage{amsthm, amssymb, bm, amsmath}
\usepackage{soul, color, graphicx}

\usepackage{todonotes}

\newtheorem{theorem}{Theorem}[]
\newtheorem{lemma}[theorem]{Lemma}

\newtheorem{remark}[theorem]{Remark}

\newtheorem{conjecture}{Conjecture}[section]

\newcommand \Cm { \mathbb{C}}
\newcommand \Dm { \mathbb{D}}
\newcommand \Rm { \mathbb{R}}
\newcommand \Nm { \mathbb{N}}
\newcommand \Sm { \mathbb{S}}
\newcommand \Zm { \mathbb{Z}}
\newcommand \zbar { \bar{z}}
\newcommand{\Vol}{ \text{Vol}}

\renewcommand \L {{\cal L}}
\newcommand \D {{\cal D}}
\newcommand \T {{\cal T}}
\newcommand \dom {{\text{dom}}}
\renewcommand \SS {{\cal{S}}}
\renewcommand \ss {{\mathfrak{s}}}

\newcommand \wtH {\widetilde{H}}
\newcommand \wtphi {\widetilde{\phi}}
\newcommand \wtpsi {\widetilde{\psi}}

\newcommand \gnk[1] {Z_{n,k}^{#1}}
\newcommand \signk[1] {\sigma_{n,k}^{#1}}
\newcommand \psink[1] {\psi_{n,k}^{#1}}

\newcommand \F[1] {#1}

\title{The $C^\infty$-isomorphism property for a class of singularly-weighted X-ray transforms}
\author{Rohit Kumar Mishra\thanks{Indian Institute of Technology Gandhinagar, Palaj, Gandhinagar, Gujarat; email:rohit.m@iitgn.ac.in} \and Fran\c{c}ois Monard\thanks{Department of Mathematics, University of California, Santa Cruz CA 95064; email:fmonard@ucsc.edu} \and Yuzhou Zou\thanks{Department of Mathematics, University of California, Santa Cruz CA 95064; email:yzou34@ucsc.edu}}
\date{}

\begin{document}
\maketitle

\begin{abstract}
    We study a one-parameter family of self-adjoint \F{normal operators for} the X-ray transform on the \F{closed} Euclidean disk $\Dm$, obtained by considering specific singularly weighted $L^2$ topologies. We first recover the well-known Singular Value Decompositions in terms of orthogonal disk (or generalized Zernike) polynomials, then prove that each such realization is an isomorphism of $C^\infty(\Dm)$. As corollaries: we give some range characterizations; we \F{show how such choices of normal operators can be expressed as functions of two distinguished differential operators}. \F{We also show that the isomorphism property also holds on a class of constant-curvature, circularly symmetric simple surfaces.} These results allow to design functional contexts where normal operators built out of the X-ray transform are provably invertible, in Fr\'echet and Hilbert spaces encoding specific boundary behavior.
\end{abstract}


\section{Introduction} \label{sec:intro}

The X-ray transform and its generalizations (on manifolds, with connection or attenuation matrix a.k.a. Higgs field) have several applications to medical, geophysical, and material imaging, while providing challenging theoretical questions in analysis and geometry, see \cite{Ilmavirta2019,Paternain2021}. Mathematically, given a convex, non-trapping Riemannian manifold with boundary $(M^d,g)$, one may model the space of unit-speed geodesics through $M$ as the inward-pointing unit vectors at the boundary
\begin{align*}
    \partial_+ SM = \{ (x,v)\in TM,\ x\in \partial M,\ g_x(v,v) = 1,\ \mu(x,v) \ge 0 \},
\end{align*}
where $\mu(x,v) = g_x(v,\nu_x)$ and $\nu_x$ is the inward-pointing normal at $x\in \partial M$. Upon equipping $M$ with its Riemannian volume form $d\Vol_g$ and $\partial_+ SM$ with its Sasaki volume form $d\Sigma^{2d-2}$, the geodesic X-ray transform $I_0\colon L^2(M, d\Vol_g) \to L^2(\partial_+ SM, \mu\ d\Sigma^{2d-2})$ is the map
\begin{align}
    I_0 f(x,v) = \int_0^{\tau(x,v)} f(\gamma_{x,v}(t))\ dt, \qquad (x,v)\in \partial_+ SM, 
    \label{eq:I0}
\end{align}
where $\gamma_{x,v}(t)$ is the unit-speed geodesic passing through $(x,v)$ and $\tau(x,v)$ is its first exit time, see \cite{Paternain2021}. We will denote $I_0^\sharp$ \F{the Hilbert space adjoint of \eqref{eq:I0}}, also known as the backprojection operator.  

The study of refined mapping properties for X-ray transform on manifolds with boundary and their normal operators, in particular accounting for boundary behavior, has recently gained attention \cite{Assylbekov2018,Monard2017,Monard2019,Monard2020a}, see also \cite{Krishnan2021,Sharafutdinov2020} for recent work on the boundary-less case. An important application is to obtain theoretical guarantees for the consistency and uncertainty quantification of statistical recovery algorithms (aimed at inverting, e.g., a noisy operator $I_0$), which rely on having a good functional framework where the operator considered is invertible. Normal operators associated with X-ray transforms in ``simple'' geometries (e.g., $I_0^\sharp I_0$) have been long known to be elliptic pseudodifferential operators in the interior \cite{Stefanov2004}, and although this provides local stability estimates in the interior, obtaining full invertibility all the way to the boundary requires deeper study. In recent works \cite{Monard2017,Monard2019,Monard2020a,Monard2019a,Mazzeo2021}, it has been observed that considering singularly weighted $L^2$ topologies on $M$ and $\partial_+ SM$ could lead to normal operators (all of the form $I_0^\sharp w_2 I_0 w_1$ \F{ with $w_{1}$ and $w_2$ weight functions on $M$ and $\partial_+ SM$, respectively}) with similar ellipticity property in the interior, while having desirable global functional properties (e.g., invertibility all the way to the boundary), and providing new ways to reconstruct $f$ from $I_0 f$. 

Depending on the choice of weights $w_{1,2}$, one becomes able to find scales of Sobolev spaces on $M$ where continuity and stability estimates can be formulated, even sometimes in an isometric manner. Those have included classical Sobolev and transmission Sobolev scales \cite{Monard2017} for the study of $I_0^\sharp I_0$, or a Sobolev scale obtained as domain spaces of powers of a degenerate elliptic operator \cite{Monard2019a} for the study of $I_0^\sharp \frac{1}{\mu} I_0$ on circularly symmetric, constant-curvature disks. One may then wonder what happens at the Fr\'echet intersection of each \F{of these} scale\F{s}, and whether one can find a setting where the mapping properties of $I_0^\sharp w_2 I_0 w_1$ can be formulated in the same scale on both the domain and the co-domain. Having the latter helps predict the regularity of the eigenfunctions, and it allows to iterate the operator, as is customarily done in certain recovery algorithms like Landweber's iteration. 

Recently, forward mapping properties have been obtained for $I_0$ and $I_0^\sharp$ on strictly geodesically convex Riemannian manifolds \cite{Mazzeo2021}, allowing one to predict the {\em index set}\footnote{By ``index set'' here we mean, the collection of $(z,k)\in \Cm\times \Nm_0$ such that a term $d^z (\log d)^k$ appears in the expansion of $f$ off the boundary, where $d$ is a boundary defining function.} of $I_0 f$ from that of $f$, and the index set of $I_0^\sharp g$ from that of $g$. This allows for a systematic understanding of the forward mapping properties of operators of the form $I_0^\sharp w_2 I_0 w_1$ on Fr\'echet subspaces of $C^\infty(M^{int})$ with polyhomogeneous conormal expansions off of $\partial M$, where the weights $w_1, w_2$ have specific singular boundary behavior. Depending on the choice of weights, there \F{may be} several (desirable or less desirable) possible outcomes: a normal operator may \F{for example} (i) map a Fr\'echet space $F$ into itself, (ii) cycle through a finite sequence of Fr\'echet spaces or (iii) yield a non-repeating chain of Fr\'echet spaces $F_1\to F_2\to F_3\dots$. The operator $I_0^\sharp I_0$ falls into the third category, while the operator $I_0^\sharp \frac{1}{\mu} I_0$ falls either into the second category as it maps the following
\begin{align*}
    C^\infty(M) + \log d\ C^\infty(M) \stackrel{I_0^\sharp \mu^{-1} I_0}{\longrightarrow} C^\infty(M) + d^{1/2} C^\infty(M) \stackrel{I_0^\sharp \mu^{-1} I_0}{\longrightarrow} C^\infty(M) + \log d\ C^\infty(M),
\end{align*}
or even the first category when $F = C^\infty(M)$.

To make this last example more general, it is observed in \cite{Mazzeo2021} that if $(d,t)$ denote boundary defining functions for $M$ and $\partial_+ SM$ (with an additional requirement on $t$ restricting its index set, see \cite[Def. 4.4]{Mazzeo2021}), for any $\gamma>-1$, the operator $I_0^\sharp t^{-2\gamma-1} I_0 d^{\gamma}$ maps $C^\infty(M)$ into itself. For purposes of inversion, one naturally asks whether the converse inclusion holds, and Conjecture 2.9 in \cite{Mazzeo2021} states: 

\begin{conjecture}\label{conj}
    Given $(M,g)$ a simple Riemannian manifold with boundary and $\gamma>-1$, there exist $(d,t)$ boundary defining functions for $M$ and $\partial_+ SM$ making $I_0^\sharp t^{-2\gamma-1} I_0 d^\gamma$ an isomorphism of $C^\infty(M)$. 
\end{conjecture}

Conjecture \ref{conj} holds true when $\gamma = -1/2$ and $(M,g)$ is a simple surface \cite{Monard2017}, or when $\gamma=0$ and $(M,g)$ is a simple geodesic disk of constant curvature \cite{Monard2019a}, although results remain sparse at this point.  

Coincidentally, for cases like the Euclidean disk, these weights are precisely those for which the Singular Value Decomposition (SVD) of the X-ray transform is known \cite{Louis1984}, although it seems that further treatment did not follow.  

In the current article, we first re-derive this SVD in more recent language associated with the X-ray transform (see Theorem \ref{thm:SVD}). Specifically, we use the method of intertwining differential operators, whose one-dimensional version was previously used in \cite{Maass1991}, here in such a way that the separation by angular harmonics is not required, see also \cite{Mishra2019,Monard2019a}. The derivation of the SVD also implies a range characterization of each of these weighted X-ray transforms, in the form of consistency/moment conditions {\it \`a la} Helgason-Ludwig or Gelfand-Graev \cite{Helgason1999,Ludwig1966,GelfandGraev1960}. 

Second, we then use this SVD to show that Conjecture \ref{conj} holds for all $\gamma>-1$ in the Euclidean unit disk (see Theorem \ref{thm:isomorphism}). The proof is similar to the case $\gamma=0$ previously treated in \cite{Monard2019a}, i.e. based on a study of the eigenfunctions (generalized Zernike polynomials in this case, denoted $\{\gnk{\gamma}\}_{n\ge 0, 0\le k\le n}$ below, equal up to scaling to the polynomials $P_{n-k,k}^\gamma$ defined in \cite{Wuensche2005}): key ingredients are (i) knowledge of their $L^\infty$ and $L^2$ norms, (ii) understanding how the derivatives $\partial_z$, $\partial_{\zbar}$ act on them, and (iii) defining a ``natural'' scale of Hilbert spaces $\{\wtH^{s, \gamma}\}_{s\ge 0}$ (defined in terms of decay along $\{\gnk{\gamma}\}_{n,k}$, see Eq. \eqref{eq:Sobolev}) with intersection $C^\infty(\Dm)$, on which $I_0^\sharp \mu^{-2\gamma-1} I_0 d^\gamma$ is tame with tame inverse. 

Though this article is self-contained, we briefly mention here key modifications from the proof in \cite{Monard2019a}. Regarding (i), unlike for $\gamma=0$, the $L^\infty$ norm cannot be sharply computed, thus we produce a slightly lossy estimate (Lemma \ref{lem:linftybound}) by expressing generalized Zernike polynomials as ``backprojected'' from Gegenbauer polynomials, whose $L^\infty$ norms are known. Regarding (ii), the action of $\partial_z$ and $\partial_{\zbar}$ on $\{\gnk{\gamma}\}$ can not be simply expressed into the same basis, though identities in \cite{Wuensche2005}  show that it is rather compact when expressed in $\{\gnk{\gamma+1}\}$ (i.e., the basis whose parameter $\gamma$ is shifted up by 1). As a result, the proof of injections of $\wtH^{s,\gamma}$ into $C^k$ spaces, and in turn that $\cap_{s\ge 0} \wtH^{s,\gamma} = C^\infty(\Dm)$ for all $\gamma>-1$, exploit properties of $\{\gnk{\gamma}\}$ as a triply-indexed family. Finally, when $\gamma\ne 0$ the singular values of $I_0^\sharp \mu^{-2\gamma-1} I_0 d^\gamma$ depend on {\em both} indices $n,k$ while for $\gamma=0$, they only depend on $n$. This difference has a couple of implications: First, while it is shown below that, like for $\gamma=0$ in \cite{Monard2019a}, there is a degenerate elliptic operator $\L_\gamma$ (defined in \eqref{eq:Lgamma}) of Kimura\footnote{see \cite{Epstein2013}} type whose eigenfunctions are also $\{\gnk{\gamma} \}_{n,k}$, the normal operator $I_0^\sharp \mu^{-2\gamma-1} I_0 d^\gamma$ is no longer in the functional calculus of $\L_\gamma$ alone, but in the {\em joint} functional calculus of $\L_\gamma$ and $\partial_\omega$ ($\omega$: angular coordinate on the disk). Second, eigenvalue asymptotics for $I_0^\sharp \mu^{-2\gamma-1} I_0 d^\gamma$ for $\gamma\ne 0$ are more complex than for $\gamma=0$, and this implies that, since $\wtH^{s,\gamma}$ can be thought of as the domain of $\L_\gamma^{s/2}$, this Sobolev scale when $\gamma\ne 0$ does not quite accurately capture Sobolev mapping properties of $I_0^\sharp \mu^{-2\gamma-1} I_0 d^\gamma$ in an isometric way. Indeed, on this scale the exponents of boundedness of the operator and its inverse are not negatives of one another.

Finally, exploiting intertwining diffeomorphisms between geometries as in \cite{Mishra2019,Monard2019a}, we also prove that Conjecture \ref{conj} holds true for all $\gamma>-1$ on simple geodesic disks of constant curvature, see Theorem \ref{thm:CCD} below. 

In the next section, we set the stage and discuss the main results before giving an outline of the remainder of the paper.

\section{Statement of the main results} \label{sec:main}

\subsection{Preliminaries}\label{ssec:prelim}

The domain will be \F{the closed unit disk} $\Dm = \{z\in \Cm, |z|\le 1\}$, equipped with the Euclidean metric and area form. The space of directed geodesics is identified with $\partial_+ S\Dm$, parameterized in fan-beam coordinates $(\beta,\alpha)\in \Sm^1\times [-\pi/2,\pi/2]$. A pair $(\beta,\alpha)$ describes the unique line passing through the boundary point $e^{i\beta}$ with ``\F{incoming}'' tangent vector $e^{i(\beta+\pi+\alpha)}$, given by 
\begin{align}
    \gamma_{\beta,\alpha}(t) = e^{i\beta} + t e^{i(\beta+\pi+\alpha)}, \qquad 0\le t\le 2\cos\alpha,
    \label{eq:gba}
\end{align}
\F{see Figure \ref{fig:setting}.}
\begin{figure}[htpb]
    \centering
    \includegraphics[height=0.24\textheight]{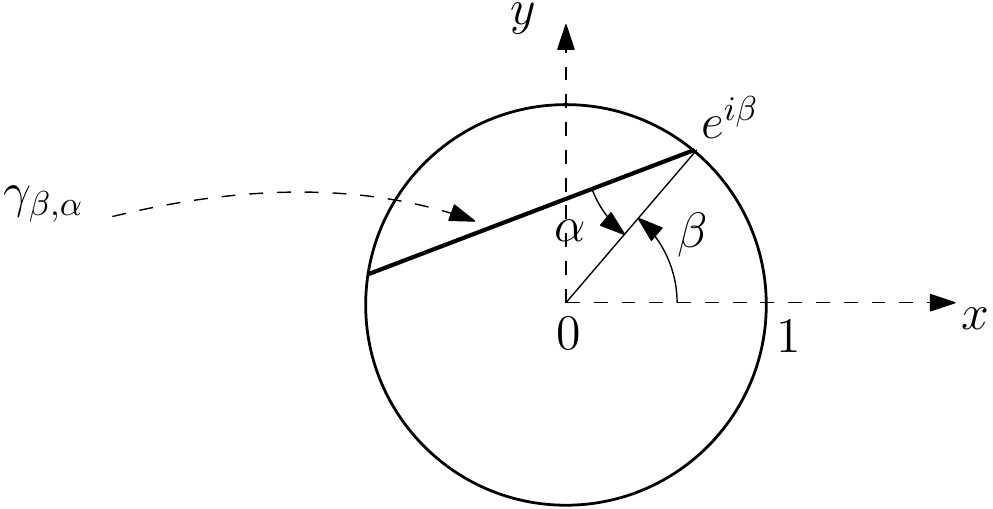}
    \caption{The setting of fan-beam coordinates.}
    \label{fig:setting}
\end{figure}
Then $\partial_+ S\Dm$ is a manifold with two boundary components $\{\alpha = \pm \pi/2\}$ and area form $d\Sigma^2 = d\beta\ d\alpha$. We equip $M$ and $\partial_+ SM$ with the following functions
\begin{align}
    d := 1-|z|^2, \quad (z\in \Dm) \quad \text{and} \quad \mu(\beta,\alpha) = \cos\alpha, \quad ( (\beta,\alpha)\in \partial_+ SM ). 
    \label{eq:bdfs}
\end{align}
\F{The function $d$ is boundary defining for the smooth structure induced by $z$, and the function $\mu$ is boundary defining for the smooth structure induced by $(\beta,\alpha)$.}

The X-ray transform $I_0$ defined in \eqref{eq:I0} and its adjoint $I_0^\sharp$ take the form 
\begin{align}
    I_0 f(\beta,\alpha) &= \int_0^{2\cos\alpha} f(\gamma_{\beta,\alpha}(t))\ dt, \qquad f\in \F{L^2}(\Dm),\ (\beta,\alpha)\in \partial_+ S\Dm, \label{eq:Xray}\\
    I_0^\sharp g(z) &= \int_{\Sm^1} g(\beta_-(z,\theta),\alpha_-(z,\theta))\ d\theta, \qquad g\in \F{L^2(\partial_+ S\Dm, \mu d\Sigma^2)},\ z\in \Dm, \label{eq:backproj}
\end{align}
where $(\beta_-,\alpha_-)(z,\theta)$ are the fan-beam coordinates of the unique line passing through $(z,e^{i\theta})$.

Our object of study is a family of singularly weighted X-ray transforms: given $\gamma>-1$, define the operator $I_0 d^\gamma (f):= I_0 (d^\gamma f)$ defined on $C(\Dm)$. By \cite[Prop. 2.7]{Mazzeo2021}, such an operator can be extended as a bounded operator in the Hilbert setting 
\begin{align}
    I_0 d^\gamma\colon L^2(\Dm, d^\gamma |dz|^2) \to L^2(\partial_+ S\Dm, \mu^{-2\gamma} d\Sigma^2), \qquad I_0 d^\gamma (f) := I_0 (d^\gamma f),
    \label{eq:weightedXray}
\end{align}
with Hilbert adjoint $I_0^\sharp \mu^{-2\gamma-1}$. For conciseness, we will drop the factors $|dz|^2$ and $d\Sigma^2$ below. 

\subsection{Main results}\label{ssec:main}

Our first result is a re-derivation of the SVD of the operator \eqref{eq:weightedXray}, in the spirit of the method of intertwining differential operators \cite{Maass1991} without separating the \F{angular} harmonic numbers. These SVDs are a special case of \cite{Louis1984}, which in higher dimensions covers the SVD of the Radon transform (integration over hyperplanes) rather than the X-ray transform. Below, the $\widehat{(\cdot)}$ notation denotes normalization in their respective spaces. On $\partial_+ S\Dm$, let us define 
\begin{align}
    \psink{\gamma} = \mu^{2\gamma+1} \frac{(-1)^n}{2\pi} e^{i(n-2k)(\beta+\alpha)} L_n^\gamma(\sin\alpha), \qquad n\ge 0, \qquad k\in \Zm,
    \label{eq:psinkgamma}
\end{align}
where $L_n^\gamma\colon [-1,1]\to \Rm$ is the $n$-th orthogonal polynomial for the weight $(1-x^2)^{\gamma+1/2}$. One may fix a choice of normalization, though the formulation of the result below is independent of it.  

For what follows, on $\Dm$, let us then define
\begin{align}
    \gnk{\gamma} := I_0^\sharp \mu^{-2\gamma-1} \psink{\gamma}, \quad n\ge 0, \quad 0\le k\le n.     
    \label{eq:Zernikegamma} 
\end{align}
We will prove below that each $\gnk{\gamma}$ is a constant multiple of $P^\gamma_{n-k,k}$ where for $m,\ell\ge 0$, $P_{m,\ell}^\gamma$ is a disk polynomial or generalized Zernike polynomial as defined in \cite{Wuensche2005}, see also Section \ref{sec:diskpoly} below. 

\begin{theorem}\label{thm:SVD}
    Fix $\gamma>-1$. 

    (1) The operator $I_0^\sharp \mu^{-2\gamma-1}\colon L^2(\partial_+ S\Dm, \mu^{-2\gamma})\to L^2(\Dm, d^\gamma)$ has kernel 
    \begin{align}
	\ker (I_0^\sharp \mu^{-2\gamma-1}) = \text{span} \left( \psink{\gamma},\quad n\ge 0,\ k<0 \text{ or } k>n \right). 
	\label{eq:keradjoint}
    \end{align}

    (2) The singular Value Decomposition of 
    \begin{align*}
	I_0 d^\gamma\colon L^2(\Dm, d^\gamma) \to L^2(\partial_+ S\Dm, \mu^{-2\gamma})/\ker(I_0^\sharp \mu^{-2\gamma-1})	
    \end{align*}
    is given by $\left(\widehat{\gnk{\gamma}}, \widehat{\psink{\gamma}}, \signk{\gamma}\right)_{n\ge 0,\ 0\le k\le n}$, where $\gnk{\gamma}$ is defined in \eqref{eq:Zernikegamma}, $\psink{\gamma}$ is defined in \eqref{eq:psinkgamma}, and 
    \begin{align}
	(\signk{\gamma})^2 &:= 2^{2\gamma+2} \pi \binom{n}{k} \frac{\Gamma(n-k+\gamma+1)\Gamma(k+\gamma+1)}{\Gamma(n+2\gamma+2)} \label{eq:SVD1}\\
	&= \frac{2^{2\gamma+2} \pi}{n+1} \frac{B(n-k+1+\gamma, k+1+\gamma)}{B(n-k+1, k+1)}, \label{eq:SVD2}	
    \end{align}
    where $B(x,y) = \frac{\Gamma(x) \Gamma(y)}{\Gamma(x+y)}$ stands for the Beta function.
\end{theorem}

Theorem \ref{thm:SVD} implies the eigenequation 
\begin{align}
    I_0^\sharp \mu^{-2\gamma-1} I_0 d^\gamma \gnk{\gamma} = (\signk{\gamma})^2 \gnk{\gamma}, \qquad n\in \Nm_0,\ 0\le k\le n. 
    \label{eq:eigeneq}
\end{align}

It is proved in \cite[Lemma 2.8]{Mazzeo2021} that an operator of the form $I_0^\sharp \mu^{-2\gamma-1} I_0 d^\gamma$ maps $C^\infty(\Dm)$ into itself (in fact, this holds for a larger class of Riemannian manifolds and boundary defining functions). \F{It is also injective, as one may see in at least two ways: for any $\gamma>-1$, $d^\gamma C^\infty(\Dm)\subset L^1(\Dm)$ and since it is well-known that $I_0$ is injective on $L^1(\Dm)$, then $I_0 d^\gamma$ is injective on $C^\infty(\Dm)$; alternatively, as a direct follow-up to Theorem \ref{thm:SVD}, since from \eqref{eq:SVD1}, $\sigma^\gamma_{n,k}\ne 0$ for all $n,k$, and \eqref{eq:eigeneq} gives a full spectral decomposition, $I_0^\sharp \mu^{-2\gamma-1} I_0 d^\gamma$ is injective on $L^2(\Dm,d^\gamma)$, in particular on the subspace $d^\gamma C^\infty(\Dm)$. }

Using \eqref{eq:eigeneq} and additional properties of generalized Zernike polynomials as defined in \eqref{eq:Zernikegamma}, we then prove the $C^\infty$-isomorphism property for our operators of interest. 

\begin{theorem}\label{thm:isomorphism}
    \F{For any $\gamma>-1$, the operator $I_0^{\sharp}\mu^{-2\gamma-1}I_0d^{\gamma}$ is an isomorphism of $C^\infty(\Dm)$.}
\end{theorem}

The case $\gamma =0$ was previously established in \cite{Monard2019a}. For $\gamma>-1$, another way to interpret Theorem \ref{thm:isomorphism} is to say that 
\begin{align*}
    I_0^{\sharp}\mu^{-2\gamma-1}I_0 \colon d^\gamma C^\infty(\Dm) \to C^\infty(\Dm)
\end{align*}
is an isomorphism. In the case $\gamma = -1/2$, this result was established in \cite{Monard2017} and was shown to hold on any simple Riemannian surface, using Boutet de Monvel calculus and $\mu$-transmission spaces introduced by H\"ormander and Grubb. 

\F{

    \paragraph{Generalization to constant curvature disks.} We now show that the $C^\infty$-isomorphism property Theorem \ref{thm:isomorphism} can be shown to hold on a family of disks of constant curvature, previously studied in \cite{Mishra2019,Monard2019a,Monard2021}.

\begin{theorem}\label{thm:CCD}
    Let $(M,g)$ be a simple geodesic disk of constant curvature, and let $\gamma>-1$. Then there exists boundary defining functions $d$ for $M$ and $t$ for $\partial_+ SM$ such that the operator $I_0^\sharp t^{-2\gamma-1} I_0 d^\gamma \colon C^\infty(M)\to C^\infty(M)$ is an isomorphism.     
\end{theorem}

Theorem \ref{thm:CCD} is proved in Section \ref{sec:CCDs}. In a nutshell, the proof is based on combining Theorem \ref{thm:isomorphism} with intertwining diffeomorphisms between the geometries of points and geodesics on such surfaces, and those of the Euclidean disk.

}

\subsection{Consequences \F{of Theorems \ref{thm:SVD} and \ref{thm:isomorphism}}}\label{ssec:corol}

\subsubsection{Range characterization}

A first obvious by-product of Theorems \ref{thm:SVD} and \ref{thm:isomorphism} is a range characterization in terms of moment conditions {\it \`a la} Helgason-Ludwig or Gelfand-Graev \cite{Helgason1999,Ludwig1966,GelfandGraev1960}. The proof is given in Section \ref{sec:proofrangeCharac}.

\begin{theorem}\label{thm:rangeCharac}
    A function $\psi\in \mu^{2\gamma+1}C_{\alpha,+,+}^\infty(\partial_+ S\Dm)$ (defined in \eqref{eq:Calpha2} below) belongs to the range of $I_0 d^\gamma$ on $C^\infty(\Dm)$ if and only if 
    \begin{align}
	\langle \psi, \psink{\gamma} \rangle_{L^2(\partial_+ S\Dm, \mu^{-2\gamma})} = 0 \qquad \text{for all} \quad \qquad n\ge 0, \quad k\in \Zm\backslash \{0, \dots, n\}. 
	\label{eq:HLGG}
    \end{align}
\end{theorem}

In the case $\gamma=0$, $I_0$ admits further characterizations in terms of operators $P_-$ (see \cite{Pestov2004}) or $C_-$ (see \cite{Monard2015a,Monard2019a}), ``naturally'' defined in terms of the scattering relation and the fiberwise Hilbert transform. It is unclear at the moment whether analogous operators exist when $\gamma\ne 0$. 

\begin{remark}
    One could also formulate similar \F{range characterizations} on Sobolev-type spaces. \F{However,} the fact that \F{the} singular values of $I_0 d^\gamma$, \F{$\{\sigma_{n,k}^\gamma, n\ge 0,\ 0\le k\le n\}$ as defined in \eqref{eq:SVD2}}, depend \F{non-trivially} on \F{the parameter} $k$ \F{unless $\gamma=0$} makes the design of Sobolev spaces look less natural than in the case $\gamma=0$. Indeed, with $\wtH^{s,\gamma}(\Dm)$ defined in \eqref{eq:Sobolev}, in view of Theorem \ref{thm:SVD}, the image $I_0 d^\gamma (\wtH^{s,\gamma}(\Dm))$ can be easily described as 
    \begin{align*}
	I_0 d^\gamma (\wtH^{s,\gamma}(\Dm)) = \left\{ \sum_{n\ge 0} \sum_{k=0}^n a_{n,k} \signk{\gamma} \widehat{\psink{\gamma}}, \qquad \sum_{n,k} (n+1+\gamma)^{2s} |a_{n,k}|^2 <\infty \right\}.	
    \end{align*}    
    However this right-hand side could neither be accurately captured in the isotropic Sobolev scale on $\partial_+ S\Dm$, nor in a Sobolev scale constructed as domain spaces of $\T_\gamma$ (defined in \eqref{eq:Tgamma} below) in a manner similar to \cite{Monard2019a}.    
\end{remark}

\subsubsection{Functional relations}

\F{For any $\gamma>-1$, the above results allow to express the normal operator $I_0^{\sharp}\mu^{-2\gamma-1}I_0d^{\gamma}$ as a function of two distinguished commuting self-adjoint differential operators.} Only in the case $\gamma = 0$ can this be reduced to a single differential operator as in \cite{Monard2019a}. Indeed, with the operator $\L_\gamma$ defined Eq. \eqref{eq:Lgamma} and $D_\omega := \frac{1}{i} \partial_\omega$ (with $\omega$ the argument coordinate on $\Dm$), we have 
\begin{align}
    \L_\gamma \gnk{\gamma} = (n+1+\gamma)^2 \gnk{\gamma}, \qquad D_\omega \gnk{\gamma} = (n-2k) \gnk{\gamma}, \qquad n\ge 0, \qquad 0\le k\le n.
\end{align}
Let us then define $\D_\gamma := \L_\gamma^{1/2} - \gamma-1$, such that $\D_\gamma \gnk{\gamma} = n \gnk{\gamma}$. Then \eqref{eq:SVD1}-\eqref{eq:SVD2} readily give us the functional relations
\begin{align}
    \begin{split}
	I_0^\sharp \mu^{-2\gamma-1} I_0 d^\gamma &= 2^{2\gamma+2} \pi \frac{\Gamma(\D_\gamma+1)}{\Gamma(\D_\gamma+2\gamma+2)} \frac{\Gamma( (\D_\gamma+D_\omega)/2 + \gamma+1)}{\Gamma((\D_\gamma+D_\omega)/2 +1)}\frac{\Gamma( (\D_\gamma-D_\omega)/2 + \gamma+1)}{\Gamma((\D_\gamma-D_\omega)/2 +1)} \\
	&= 2^{2\gamma+2} \pi (\D_\gamma +1)^{-1} \frac{B((\D_\gamma+D_\omega)/2+1+\gamma, (\D_\gamma-D_\omega)/2+1+\gamma)}{B((\D_\gamma+D_\omega)/2+1, (\D_\gamma-D_\omega)/2+1)}.  
    \end{split}    
    \label{eq:funcrel}
\end{align}
Similar functional relations have been derived on symmetric spaces \cite{Helgason2010} or surfaces of constant negative curvature \cite{Guillarmou2017}. In most if not all previous cases where this is possible, integral-geometric operators are often expressed as functions of a {\em single} elliptic differential operator.

\subsubsection*{Outline} 

The remainder of the article is organized as follows. We first recall notation associated with generalized disk polynomials in Section \ref{sec:prelim2}, and introduce two important families of differential operators $\L_\gamma$ in Sec. \ref{sec:Lgamma}, and $\T_\gamma$ in Sec. \ref{sec:Tgamma}. We then derive the SVD of $I_0 d^\gamma$ (Theorem \ref{thm:SVD}) in Section \ref{sec:proofThm1}, first providing a rough outline in Section \ref{sec:IDOs} based on the method of intertwining differential operators, then proving all auxiliary lemmas in Section \ref{sec:lemmas1}. We then prove the $C^\infty$-isomorphism property on the Euclidean disk (Theorem \ref{thm:isomorphism}) in Section \ref{sec:isomorphism}, also providing an outline in Section \ref{sec:outline2} before proving all auxiliary lemmas in Section \ref{sec:lemmas2}, as well as Theorem \ref{thm:rangeCharac} in Section \ref{sec:proofrangeCharac}. Finally, we provide the $C^\infty$-isomorphism property for geodesic disks of constant curvature in Section \ref{sec:CCDs}.

\section{Disk polynomials and some important differential operators}\label{sec:prelim2}

\subsection{Disk polynomials, or generalized Zernike polynomials}\label{sec:diskpoly}

Our choice of notation is a combination of \cite{Kazantsev2004} which is well-adapted to transport problems and Cauchy-Riemann systems, and \cite{Wuensche2005}, which compiles several important facts for disk polynomials. 

Following \cite{Wuensche2005}, for $\gamma\in \Rm$ and $m,\ell\in \Nm_0$, we introduce 
\begin{align}
    P_{m,\ell}^\gamma(z,\zbar) = \frac{\ell! \gamma!}{(\ell+\gamma)!} z^{m-\ell} P_\ell^{(\gamma,m-\ell)} (2|z|^2-1) = \frac{m! \gamma!}{(m+\gamma)!} \zbar^{\ell-m} P_m^{(\gamma,\ell-m)} (2|z|^2-1),
    \label{eq:Pml}
\end{align}
of degree $m+\ell$, where $P_n^{(a,b)}$ denotes the Jacobi polynomials. This family is orthogonal on $L^2(\Dm, (1-|z|^2)^\gamma)$. Below, to make it similar to \cite{Kazantsev2004,Monard2019a}, we will only look at it through the indexing 
\begin{align*}
    (n,k) \to P_{n-k,k}^\gamma, \qquad n\in \Nm_0, \qquad 0\le k\le n.
\end{align*}
Then $n$ plays the role of the degree, $P^\gamma_{n,0}$ is proportional to $z^n$, the polar dependency of $P_{n-k,k}^\gamma$ is $e^{i(n-2k)\omega}$, and the action $(n,k)\to (n,k+1)$ is a form of Beurling transform. One has 
\begin{align*}
    P^\gamma_{n-k,k}(z,\zbar) = p_{n-k,k}^\gamma z^{n-k}\zbar^k + l.o.t, 
\end{align*}
and the following quantities are tabulated: 
\begin{align}
    \|P_{n-k,k}^\gamma\|^2_{L^2(\Dm, d^\gamma)} = \frac{\pi}{n+\gamma+1} \frac{(n-k)! \gamma! k! \gamma !}{ (k+\gamma)! (n-k+\gamma)!}, \qquad p_{n-k,k}^\gamma = \frac{\gamma! (n+\gamma)!}{(n-k+\gamma)! (k+\gamma) !},
    \label{eq:norms}
\end{align}
where for $x>-1$, $x!$ is short for $\Gamma(x+1)$. Upon defining in polar coordinates $(z=\rho e^{i\omega})$
\begin{align}
    \L_\gamma := -(1-\rho^2) \partial_\rho^2 - (\rho^{-1}-(3+2\gamma)\rho)\partial_\rho -\rho^{-2} \partial_\omega^2 + (\gamma+1)^2 id, 	
    \label{eq:Lgamma}
\end{align}
a second-order differential operator with smooth coefficients on $\Dm$, we deduce from \cite[Eq. (4.3) and (4.6)]{Wuensche2005} the eigenequations
\begin{align}
    \partial_\omega P_{n-k,k} = i(n-2k)P_{n-k,k}, \quad \L_\gamma P_{n-k,k} = (n+1+\gamma)^2 P_{n-k,k}, \quad n\in \Nm_0, \quad 0\le k\le n.
    \label{eq:eigenequations}
\end{align}

\subsection{A family of degenerate elliptic operators on $\Dm$}\label{sec:Lgamma}

An immediate computation also gives the following expression in polar coordinates
\begin{align}
    \L_\gamma = - \rho^{-1} (1-\rho^2)^{-\gamma} \partial_\rho \left( \rho (1-\rho^2)^{\gamma+1} \partial_\rho \right)  - \rho^{-2}\partial_\omega^2 + (1+\gamma)^2 id. 
    \label{eq:Lgamma2}
\end{align}
This last expression makes it easy to check that $(\L_{\gamma},C^\infty(\Dm))$ is a symmetric, positive, unbounded operator on $L^2(\Dm, d^\gamma)$. Indeed, an integration by parts with vanishing boundary term gives, for any $u\in C^\infty(\Dm)$, 
\begin{align}
    (\L_\gamma u, u)_{L^2(\Dm, d^\gamma)} = \int_\Dm \left[ d |\partial_\rho u|^2 + \rho^{-2} |\partial_\omega u|^2 + (1+\gamma)^2 |u|^2\right] d^{\gamma} \rho\ d\rho\ d\omega \ge (1+\gamma)^2 \|u\|_{L^2(\Dm, d^\gamma)}^2.
    \label{eq:IBPL}
\end{align}

\begin{theorem} \label{thm:esssa}
    The operator $(\L_{\gamma},C^\infty(\Dm))$ densely defined on $L^2(\Dm,d^\gamma)$ is essentially self-adjoint.     
\end{theorem}

\begin{proof}
    Using \cite[Theorem 9.8]{Helffer2013}, is suffices to show that the range of $\L_{\gamma} + b$ is dense in $L^2(\Dm,d^\gamma)$ for some $b>0$. We prove this by showing that polynomials of arbitrary order are in the range of $\L_\gamma + b$ for any $b>0$. In fact, by \eqref{eq:IBPL} we have 
    \begin{align*}
	((\L_{\gamma}+b)u,u)_{L^2(\Dm, d^\gamma)} \ge ((1+\gamma)^2+b)\|u\|_{L^2(\Dm, d^\gamma)},	
    \end{align*}
    and hence $\L_{\gamma}+b$ is injective on $C^{\infty}(\Dm)$ for any $b>0$ (in fact, for any $b>-(1+\gamma)^2$). As such, if we let
    \[\mathcal{P}_n = \{\text{polynomials of degree at most }n\text{ on }\Dm\} = \text{span }\{z^\alpha\zbar^\beta\,:\,\alpha,\beta\in\Nm_0,\ \alpha+\beta\le n\},\]
    it suffices to show that $\L_{\gamma}+b$ maps $\mathcal{P}_n$ into $\mathcal{P}_n$ for each $n\in\mathbb{N}$, since the injectivity of $\L_{\gamma}+b$ and the finite-dimensionality of each $\mathcal{P}_n$ would then imply that $\L_{\gamma}+b$ is a linear isomorphism on each $\mathcal{P}_n$. Recalling the formula
    \[\Delta=\partial_{\rho}^2+\rho^{-1}\partial_{\rho}+\rho^{-2}\partial_{\omega}^2,\]
    of the Laplacian in polar coordinates, it follows from \eqref{eq:Lgamma} that
    \[\L_{\gamma}+b = -(1-\rho^2)\Delta+(3+2\gamma)\rho\partial_{\rho}-\partial_{\omega}^2+((\gamma+1)^2+b)id.\]
    For $f\in\mathcal{P}_n$, we have $\Delta f\in \mathcal{P}_{n-2}$, and hence $-(1-\rho^2)\Delta f\in\mathcal{P}_n$. Since $\rho\partial_{\rho}$, $\partial_{\omega}^2$, and $id$ also map $\mathcal{P}_n$ into itself, it follows that $\L_{\gamma}+b$ maps $\mathcal{P}_n$ into itself, as desired.
\end{proof}

Theorem \ref{thm:esssa} is of course no surprise, since the eigenequation \eqref{eq:eigenequations} and the fact that the family \eqref{eq:Pml} is a complete set in $L^2(\Dm,d^\gamma)$ gives a full spectral decomposition. 

\subsection{A family of differential operators on $\partial_+ S\Dm$} \label{sec:Tgamma}

\paragraph{The operator $T$ and symmetries in $\partial_+ S\Dm$.} On $\partial S\Dm$, recall the operator $T = \partial_\beta - \partial_\alpha$, as well as the scattering and antipodal scattering relations
\begin{align}
    \SS(\beta,\alpha) := (\beta+\pi+2\alpha, \pi-\alpha), \qquad \SS_A(\beta,\alpha) := (\beta+\pi+2\alpha, -\alpha),
    \label{eq:Scatrel}
\end{align}
and let $A_{+/-}$ the extension operators by $\SS$-oddness/evenness from $\partial_+ S\Dm$ to $\partial S\Dm$. Natural domains of definition for $A_\pm$ are
\begin{align}
    C_{\alpha,\pm}^\infty(\partial_+ S\Dm) = \{ f\in C^\infty(\partial_+ S\Dm):\ A_\pm f \in C^\infty(\partial S\Dm) \},
    \label{eq:Calpha}
\end{align}
previously studied in \cite{Pestov2005,Mishra2019,Mazzeo2021}. One encodes one more symmetry by splitting each space into
\begin{align}
    \begin{split}
	C^\infty_{\alpha,\pm} (\partial_+ S\Dm) &= C^\infty_{\alpha,\pm,+} (\partial_+ S\Dm) \oplus C^\infty_{\alpha,\pm,-} (\partial_+ S\Dm), \qquad \text{where} \\
	C_{\alpha,\sigma_1,\sigma_2}^\infty(\partial_+ S\Dm) &:= \{ f\in C^\infty_{\alpha,\sigma_1}, \quad \SS_A^* f = \sigma_2 f   \}, \qquad \sigma_1,\sigma_2 \in \{+,-\}.	
    \end{split}    
    \label{eq:Calpha2}
\end{align}
Now, we have the obvious relations
\begin{align*}
    T \circ \SS^* = - \SS^* \circ T, \qquad T\circ \SS_A^* = - \SS_A^* \circ T \qquad \text{on } C^\infty(\partial S\Dm),
\end{align*}
and this immediately implies 
\begin{align}
    T(C_{\alpha,\sigma_1,\sigma_2}^\infty(\partial_+ S\Dm)) \subset C_{\alpha,-\sigma_1,-\sigma_2}^\infty(\partial_+ S\Dm), \qquad \sigma_1,\sigma_2 \in \{+,-\}.
    \label{eq:propT}
\end{align}
Since $\mu\in C_{\alpha,-,+}^\infty(\partial_+ S\Dm)$, we also have the property (see also \cite[Lemma 4.3]{Mazzeo2021} with, in the present case, $\mu = \tau/2$)
\begin{align}
    \mu C_{\alpha,+,\pm}^\infty(\partial_+ S\Dm) = C^\infty_{\alpha,-,\pm}(\partial_+ S\Dm). 
    \label{eq:propmu}
\end{align}

\paragraph{The operator $\T_\gamma$.} Now for $\gamma\in \Rm$, define 
\begin{align}
    \T_\gamma := - \mu^{2\gamma} T \mu^{-2\gamma} T + \gamma^2\ id,
    \label{eq:Tgamma}
\end{align}
acting unboundedly on $L^2(\partial_+ S\Dm, \mu^{-2\gamma})$. We claim that 
\begin{align}
    \T_\gamma( \mu^{2\gamma+1} C_{\alpha,+,+}^\infty(\partial_+ S\Dm)) \subset \mu^{2\gamma+1} C_{\alpha,+,+}^\infty(\partial_+ S\Dm).
    \label{eq:claimTgamma}
\end{align}
Together with the fact that $\mu^{2\gamma+1} C_{\alpha,+,+}^\infty(\partial_+ S\Dm)$ is dense in 
\begin{align}
    L_+^2(\partial_+ S\Dm, \mu^{-2\gamma}) := L_+^2(\partial_+ S\Dm, \mu^{-2\gamma}) \cap \ker (Id - \SS_A),
    \label{eq:L2plus}
\end{align}
Equation \eqref{eq:claimTgamma} allows us to take $\mu^{2\gamma+1} C_{\alpha,+,+}^\infty(\partial_+ S\Dm)$ as domain for $\T_\gamma$.

\begin{proof}[Proof of \eqref{eq:claimTgamma}]
    Considering the identity
    \begin{align}
	\mu^{-2\gamma-1} \T_\gamma \mu^{2\gamma+1} = - T^2 - 2(\gamma+1) \tan\alpha T + (\gamma+1)^2 id,
	\label{eq:Tgammaid}
    \end{align}
    to prove \eqref{eq:claimTgamma}, it is enough to show that (i) $T^2 (C_{\alpha,+,+}^\infty(\partial_+ S\Dm))) \subset C_{\alpha,+,+}^\infty(\partial_+ S\Dm))$, and that (ii) $\tan\alpha T (C_{\alpha,+,+}^\infty(\partial_+ S\Dm))) \subset C_{\alpha,+,+}^\infty(\partial_+ S\Dm))$. The statement (i) follows from applying \eqref{eq:propT} twice. As for (ii): with $\tan \alpha = \frac{T\mu}{\mu}$, and with $T\mu \in C_{\alpha,+,-}^\infty$, we have 
    \begin{align*}
	\tan\alpha T (C_{\alpha,+,+}^\infty) \stackrel{\eqref{eq:propT}}{\subset} \frac{T\mu}{\mu} C_{\alpha,-,-}^\infty \stackrel{\eqref{eq:propmu}}{=} T\mu C_{\alpha,+,-}^\infty \subset C_{\alpha,+,+}^\infty,
    \end{align*}
    hence the claim holds.    
\end{proof}

Now for $\phi,\psi\in \mu^{2\gamma+1} C_{\alpha,+,+}^\infty$, which we write $\phi = \mu^{2\gamma+1} \wtphi$ for some $\wtphi\in C_{\alpha,+,+}^\infty$ and similar for $\psi$, we first write
\begin{align*}
    \T_\gamma (\mu^{2\gamma+1}\wtphi) &\!\!\stackrel{\eqref{eq:Tgammaid}}{=} \mu^{2\gamma+1} (- T^2 - 2(\gamma+1) \tan\alpha T + (\gamma+1)^2) \wtphi \\
    &= -\mu^{-1} T(\mu^{2(\gamma+1)} T) \wtphi + (\gamma+1)^2 \mu^{2\gamma+1} \wtphi,
\end{align*} 
and hence
\begin{align}
    \left\langle \T_\gamma \phi, \psi \right\rangle_{\mu^{-2\gamma}} &= \int_{\partial_+ S\Dm} (-T (\mu^{2(\gamma+1)} T\wtphi) \overline{\wtpsi} + (\gamma+1)^2 \mu ^{2(\gamma+1)} \wtphi \overline{\wtpsi} )\ d\beta\ d\alpha \nonumber \\
    &= \left\langle T\wtphi, T\wtpsi \right\rangle_{L^2_+(\partial_+ S\Dm, \mu^{2(\gamma+1)})} + (\gamma+1)^2 \left\langle \wtphi, \wtpsi \right\rangle_{L^2_+(\partial_+ S\Dm, \mu^{2(\gamma+1)})},
    \label{eq:IBPTgamma}
\end{align}
where we have used that the integration by parts has no boundary term since for $\gamma>-1$, $\mu^{2(\gamma+1)}=0$ on $\partial_0 SM$, and all other terms are smooth. We conclude that for $\gamma>-1$, the operator $(\T_\gamma, \mu^{2\gamma+1} C_{\alpha,+,+}^\infty(\partial_+ S\Dm))$ is also symmetric and positive on $L_+^2(\partial_+ S\Dm, \mu^{-2\gamma})$ defined in \eqref{eq:L2plus}.

Similarly to Theorem \ref{thm:esssa} for $\L_\gamma$, we have 
\begin{theorem}\label{thm:esssa2}
    The operator $(\T_\gamma, \mu^{2\gamma+1} C_{\alpha,+,+}^\infty(\partial_+ S\Dm))$ densely defined on $L_+^2(\partial_+ S\Dm, \mu^{-2\gamma})$ is essentially self-adjoint.
\end{theorem}

\begin{proof}
    We use again \cite[Theorem 9.8]{Helffer2013}, and show that $\T_\gamma + b$ has dense range, by showing that $\mu^{-2\gamma-1} \T_\gamma \mu^{2\gamma+1} + b$ has dense range in $C_{\alpha,+,+}^\infty$ and thus in $L_+^2(\partial_+ S\Dm, \mu^{2(\gamma+1)})$. First note that the identity \eqref{eq:IBPTgamma} implies that $\T_\gamma+b$ is injective on $\mu^{2\gamma+1}C_{\alpha,+,+}^\infty$ for all $b>-(\gamma+1)^2$, and thus $\mu^{-2\gamma-1} \T_\gamma \mu^{2\gamma+1} + b$ is injective on $C_{\alpha,+,+}^\infty$. Similarly to Theorem \ref{thm:esssa}, we find finite linear subspaces of $C_{\alpha,+,+}^\infty$ which are stable under $\mu^{-2\gamma-1} \T_\gamma \mu^{2\gamma+1} + b$ on which it becomes an isomorphism. Since finite sums of these subspaces are dense in $C_{\alpha,+,+}^\infty$, the result will follow. Recall from \cite[Proposition 7]{Mishra2019} that $C_{\alpha,+,+}^\infty (\partial_+ S\Dm)$ is described in terms of expansions with rapid decay along the span
    \begin{align*}
	\left\langle \phi_{p,q} := e^{ip(\beta+\alpha)} (e^{i(p-2q)\alpha} + (-1)^p e^{-i(p-2q)\alpha}), \quad p,q\in \Zm\right\rangle, 
    \end{align*}
    with the redundancy $\phi_{p,p-q} = (-1)^p \phi_{p,q}$. Now for $m\in \Nm_0$ and $k\in \Zm$, let us define
    \begin{align*}
	{\cal P}_{m,k,+} &:= \left\{ e^{i(2m-2k)(\beta+\alpha)} \sum_{\ell=0}^m a_\ell \sin^{2\ell} \alpha,\ a_0, \dots, a_m \in \Cm \right\} = \left\langle \phi_{2m-2k,\ell-k},\ 0\le \ell\le m \right\rangle,  \\
	{\cal P}_{m,k,-} &:= \left\{ e^{i(2m+1-2k)(\beta+\alpha)} \sum_{\ell=0}^m a_\ell \sin^{2\ell+1} \alpha,\ a_0, \dots, a_m \in \Cm \right\} = \left\langle \phi_{2m+1-2k,\ell-k},\ 0\le \ell\le m \right\rangle.
    \end{align*}
    Using that for any $p$, $T^2(\sin^p\alpha) = -p^2 \sin^p \alpha + p(p-1)\sin^{p-2}\alpha$ and $\tan \alpha T (\sin^p \alpha) = - p \sin^p \alpha$, it is easily seen that these spaces are stable under $\mu^{-2\gamma-1} \T_\gamma \mu^{2\gamma+1} + b$. This completes the proof.
\end{proof}

\section{Singular Value Decomposition - proof of Theorem \ref{thm:SVD}} \label{sec:proofThm1}

\subsection{Outline of the proof: intertwining differential operators} \label{sec:IDOs}

We will follow the overall method of intertwining differential operators given in \cite{Maass1991}, though revisited without separating angular harmonics, and emphasizing the ``disk'' viewpoint, see also \cite{Mishra2019}. 

A first key observation is the following intertwining relations: 

\begin{lemma}\label{lem:intertwiners}
    Let $\gamma\in \Rm$, and $\L_\gamma$, $\T_\gamma$ defined in \eqref{eq:Lgamma} and \eqref{eq:Tgamma}. Then we have that 

    (1) $I_0^\sharp \mu^{-2\gamma-1} ( \mu^{2\gamma+1} C_{\alpha,+,+}^\infty (\partial_+ S\Dm)) \subset C^\infty(\Dm)$.

    (2) On $C_{\alpha,+,+}^\infty(\partial_+ SM)$, we have 
    \begin{align}
	I_0^\sharp \mu^{-2\gamma-1} \circ \T_\gamma = \L_\gamma \circ I_0^\sharp \mu^{-2\gamma-1} \qquad \text{and} \qquad I_0^\sharp \mu^{-2\gamma-1} \circ \partial_\beta = \partial_\omega \circ I_0^\sharp \mu^{-2\gamma-1}.     
	\label{eq:intertwiners}
    \end{align}    
\end{lemma}

We then use the following lemma. 

\begin{lemma}\label{lem:IDOclosure}
    For $i=1,2$, let $(D_i, \dom(D_i))$, densely defined on a Hilbert space $(H_i, \|\cdot\|_i)$, be two essentially self-adjoint operators. Suppose further that an operator $A\in {\cal B} (H_1,H_2)$ satisfies: 
    \begin{align}
	(i)\quad A(\dom(D_1)) \subset \dom (D_2) \qquad (ii)\quad A\circ D_1 = D_2 \circ A \text{ on } \dom(D_1). 	
	\label{eq:IDOcondition}
    \end{align}   
    Then $A(\dom(\overline{D_1})) \subset \dom (\overline{D_2})$ and $A\circ \overline{D_1} = \overline{D_2} \circ A$ on $\dom(\overline{D_1})$. 
\end{lemma}

Fixing $\gamma\in \Rm$ and considering $A = I_0^\sharp \mu^{-2\gamma-1}$, $D_1 = \T_\gamma$ with domain $\mu^{2\gamma+1} C_{\alpha,+,+}^\infty (\partial_+ S\Dm)$, and $D_2 = \L_\gamma$ with domain $C^\infty(\Dm)$, and using Theorems \ref{thm:esssa} and \ref{thm:esssa2}, we deduce that the intertwining property extends to the closure of $\T_\gamma$ and $\L_\gamma$. Similarly with $D_1 = \partial_\beta$ and $D_2 = \partial_\omega$.  

Next, the operator $I_0^\sharp \mu^{-2\gamma-1}$ intertwines the pair $(\overline{\T_\gamma}, \overline{D_\beta})$ with the pair $(\overline{L_\gamma}, \overline{D_\omega})$. We then have the following: 

\begin{lemma}\label{lem:spectrum1}
    The pair of commuting self-adjoint operators $(\overline{\T_\gamma}, \overline{D_\beta})$ satisfies
    \begin{align*}
	(\overline{\T_\gamma}, \overline{D_\beta}) \psink{\gamma} = ((n+1+\gamma)^2, n-2k) \psink{\gamma}, \qquad n\ge 0, \qquad k\in \Zm,
    \end{align*}
    where $\psink{\gamma}$ is defined in Eq. \eqref{eq:psinkgamma}. 
\end{lemma}

In particular, this joint pair of operators has simple spectrum and, upon computing $I_0^\sharp \mu^{-2\gamma-1}$ on the eigenbasis, we obtain the following simple dichotomy: for any $n\ge 0$ and $k\in \Zm$, $I_0^\sharp \mu^{-2\gamma-1} \psink{\gamma}$ is either (i) zero, or (ii) a non-trivial eigenfunction of $(\overline{\L_\gamma}, \overline{D_\omega})$ with eigenvalue $( (n+1+\gamma)^2, n-2k)$. We first show the following lemma, thus proving \eqref{eq:keradjoint}. 

\begin{lemma} \label{lem:kerI0star}
    For any $n\ge 0$, we have $I_0^\sharp \mu^{-2\gamma-1} \psink{\gamma} =0$ if and only if $k<0$ or $k>n$. 
\end{lemma}

On the orthocomplement of $\ker (I_0^\sharp \mu^{-2\gamma-1})$, all remaining functions are associated with simple eigenvalues of the pair $(\overline{\L_\gamma}, \overline{D_\omega})$, and as such form an orthogonal system. Upon calling $\gnk{\gamma} = I_0^\sharp \mu^{-2\gamma-1} \psink{\gamma}$ for $n\ge 0$ and $0\le k\le n$, we deduce that since both families are orthogonal, the SVD is simply deduced by normalizing them in their respective spaces, thereby obtaining 
\begin{align*}
    I_0^\sharp \mu^{-2\gamma-1} \widehat{\psink{\gamma}} = \signk{\gamma}\ \widehat{\gnk{\gamma}}, \quad n\ge 0, \quad 0\le k\le n,
\end{align*}
where $\signk{\gamma} = \|\gnk{\gamma}\|_{L^2(\Dm,d^\gamma)}/ \|\psink{\gamma}\|_{L^2(\partial_+ S\Dm, \mu^{-2\gamma})}$ are computed via bookkeeping combined with what is known about orthogonal polynomials in \cite{Wuensche2005}. By passing to the adjoint, this gives the SVD of $I_0 d^\gamma$ as given in the statement of Theorem \ref{thm:SVD}. 

\subsection{Proofs of intermediate lemmas}\label{sec:lemmas1}

\begin{proof}[Proof of Lemma \ref{lem:intertwiners}]
    First, the fact that $I_0^\sharp \mu^{-2\gamma-1} (\dom (\T_\gamma)) \subset \dom (\L_\gamma)$ is a direct consequence of the fact that $I_0^\sharp (C_{\alpha,+,+}^\infty (\partial_+ S\Dm)) \subset C^\infty(\Dm)$, first established in \cite{Pestov2005}. On to the intertwining relations, they will be mainly derived from the following three intertwining relations in \eqref{eq:intertwiners_paper}, see \cite{Monard2019a}. Defining the differential operator in polar coordinates $(z = \rho e^{i\omega})$
    \begin{align}
	L := (1-\rho^2) \partial_\rho^2 + (\rho^{-1}-3\rho)\partial_\rho + \rho^{-2} \partial_\omega^2,
	\label{eq:Lop}
    \end{align}
    and on $\partial_+ S\Dm$, $T:= \partial_\beta-\partial_\alpha$, $\mu = \cos \alpha$, so that $T\mu = \sin\alpha$. Then the intertwining relations hold: 
    \begin{align}
	L\circ I_0^\sharp = I_0^\sharp \circ (T^2 + 2\tan\alpha T), \qquad - \rho\partial_\rho \circ I_0^\sharp = I_0^\sharp \circ (\tan\alpha T), \qquad \partial_\beta \circ I_0^\sharp = I_0^\sharp \circ \partial_\omega, 
	\label{eq:intertwiners_paper}
    \end{align}
    so that the right equation in \eqref{eq:intertwiners} is satisfied. On to proving the left equation in \eqref{eq:intertwiners}, we write
    \begin{align*}
	I_0^\sharp\circ \mu^{-2\gamma-1} \T_\gamma \mu^{2\gamma+1} &\stackrel{\eqref{eq:Tgammaid}}{=} I_0^\sharp \circ (- T^2 - 2(\gamma+1) \tan\alpha T + (\gamma+1)^2 id) \\ 
	&\stackrel{\eqref{eq:intertwiners}}{=} (-L -2\gamma \rho\partial \rho + (\gamma+1)^2 id)\circ I_0^\sharp. 
    \end{align*}
    Upon right-multiplying by $\mu^{-2\gamma-1}$ and using that $\L_\gamma = -L -2\gamma \rho\partial \rho + (\gamma+1)^2 id$, the proof is complete. 
\end{proof}

\begin{proof}[Proof of Lemma \ref{lem:IDOclosure}] Suppose $u\in \dom(\overline{D_1})$, then there is a sequence $u_n\in \dom (D_1)$ such that $u_n\to u$ in $H_1$ and $D_1 u_n \to v$ in $H_1$, with $D_1 u = v$. Let $x_n := Au_n \in \dom(D_2)$. Since $A$ is continuous, $x_n$ converges to $Au$ in $H_2$ and $y_n := D_2 x_n = D_2 A u_n = A D_1 u_n$ converges to $A v$ in $H_2$. This means that $Au\in \dom(\overline{D_2})$, and since $\overline{D_2}$ is closed, $D_2 Au = Av = AD_1 u$. Hence the result.     
\end{proof}

\begin{proof}[Proof of Lemma \ref{lem:spectrum1}] Given the form of $\psink{\gamma}$, the relation $D_\beta \psink{\gamma} = (n-2k) \psink{\gamma}$ is immediate. On to $\T_\gamma$, 
    \begin{align}
	\T_\gamma \psink{\gamma} &\stackrel{\eqref{eq:Tgammaid}}{=} \frac{(-1)^n}{4\pi} \mu^{-2\gamma-1}[-T^2 - 2(\gamma+1) \tan\alpha T + (\gamma+1)^2 id] e^{i(n-2k)(\beta+\alpha)} L_n^\gamma(\sin \alpha) \nonumber \\
	&=  \frac{(-1)^n}{4\pi} \mu^{-2\gamma-1} e^{i(n-2k)(\beta+\alpha)} [-T^2 - 2(\gamma+1) \tan\alpha T + (\gamma+1)^2 id]  L_n^\gamma(\sin \alpha), \label{eq:temp}
    \end{align}
    since $T(f(\beta+\alpha)) =0$ for any function $f$. Now, $L_n^\gamma$ is proportional to the Gegenbauer polynomial $C_n^{(\gamma+1)}$, which satisfies the eigenvalue equation
    \begin{align*}
	-(1-x^2) (C_n^{(\gamma+1)})'' + (2\gamma+3) x (C_n^{(\gamma+1)})' = n(n+2\gamma+2) C_n^{(\gamma+1)},
    \end{align*}
    hence so does $L_n^\gamma$. This implies that 
    \begin{align*}
	(- T^2 - 2(\gamma+1)\tan\alpha T)  L_n^\gamma(\sin\alpha) = n(n+2\gamma+2) L_n^\gamma(\sin\alpha),
    \end{align*}
    and the relation $\T_\gamma \psink{\gamma} = (n+1+\gamma)^2 \psink{\gamma}$ follows upon plugging this into \eqref{eq:temp}. 
\end{proof}

\begin{proof}[Proof of Lemma \ref{lem:kerI0star}] We work in polar coordinates where $z = \rho e^{i\omega}$. Recall that, if $(\beta_-,\alpha_-) (\rho e^{i\omega}, \theta)\in \partial_+ S\Dm$ denote the fan-beam coordinates of the unique geodesic passing through $(\rho e^{i\omega}, \theta)$, we have 
    \begin{align*}
	\beta_- + \alpha_- + \pi = \theta, \qquad \sin \alpha_- (\rho e^{i\omega}, \theta) = \sin \alpha_- (\rho, \theta-\omega) = -\rho \sin (\theta-\omega), 
    \end{align*}
    see e.g. \cite[Section 3]{Monard2019a}. Using this, we compute
    \begin{align*}
	I_0^\sharp \mu^{-2\gamma-1} \psink{\gamma} (\rho e^{i\omega}) &= \frac{(-1)^n}{2\pi} I_0^\sharp [e^{i(n-2k)(\beta+\alpha)} L_n^\gamma(\sin\alpha)] (\rho e^{i\omega}) \\
	&= \frac{1}{2\pi} \int_{\Sm^1} e^{i(n-2k)\theta} L_n^\gamma (\sin \alpha_- (\rho e^{i\omega}, \theta))\ d\theta \\
	&= \frac{1}{2\pi}  \int_{\Sm^1} e^{i(n-2k)\theta} L_n^\gamma (-\rho \sin (\theta-\omega) )\ d\theta.
    \end{align*}
    Now using that $L_n^\gamma$ is a polynomial of degree $n$, we see that the Fourier content of $\theta\mapsto L_n^\gamma (-\rho \sin(\theta-\omega))$ is supported in $e^{-in\theta}, \dots, e^{in\theta}$. As a result, if $k<0$ or $k>n$, then $n-2k<-n$ or $n-2k>n$, and in either case, the integral defining $I_0^\sharp \mu^{-2\gamma-1} \psink{\gamma}$ vanishes. 

    Finally, we show that for $n\ge 0$ and $0\le k\le n$, $\gnk{\gamma} := I_0^\sharp \mu^{-2\gamma-1} \psink{\gamma} \ne 0$. From the identity $-\rho \sin(\theta-\omega) = \frac{i}{2} (\zbar e^{i\theta} - ze^{-i\theta})$ we have 
    \begin{align}
	\gnk{\gamma} = \frac{1}{2\pi} \int_{\Sm^1} e^{i(n-2k)\theta} L_n^\gamma \left( \frac{i}{2} (\zbar e^{i\theta} - z e^{-i\theta}) \right)\ d\theta,
	\label{eq:lntognk}
    \end{align}
    and thus $\gnk{\gamma}$ can be viewed as the $(2k-n)$-th Fourier coefficient of $\theta\mapsto L_n^\gamma \left( \frac{i}{2} (\zbar e^{i\theta} - z e^{-i\theta}) \right)$. In other words
    \begin{align*}
	L_n^\gamma \left( \frac{i}{2} (\zbar e^{i\theta} - z e^{-i\theta}) \right) = \sum_{k=0}^n e^{-i(n-2k)\theta} \gnk{\gamma}. 
    \end{align*}
    This last relation easily determines the top-degree terms of each $\gnk{\gamma}$: from $L_n^\gamma(x) = \ell_n^\gamma x^n + l.o.t\dots$, we get 
    \begin{align*}
	L_n^\gamma \left( \frac{i}{2} (\zbar e^{i\theta} - z e^{-i\theta}) \right) &= \ell_n^\gamma \left( \frac{i}{2} (\zbar e^{i\theta} - z e^{-i\theta}) \right)^n + l.o.t (z,\zbar) \\
	&= \frac{\ell_n^\gamma}{(2i)^n} \sum_{k=0}^n e^{-i(n-2k)\theta} (-1)^k \binom{n}{k} \zbar^k z^{n-k} + l.o.t (z,\zbar),
    \end{align*}
    and hence $\gnk{\gamma} = g_{n,k}^\gamma \zbar^k z^{n-k} + l.o.t(z,\zbar)$, where 
    \begin{align}
	g_{n,k}^\gamma := (-1)^k \frac{\ell_n^\gamma}{(2i)^n}  \binom{n}{k}.
	\label{eq:gnkgamma}
    \end{align} 
    In particular, $I_0^\sharp \mu^{-2\gamma-1} \psink{\gamma}$ is non-trivial whenever $n\ge 0$ and $0\le k\le n$. Lemma \ref{lem:kerI0star} is proved.
\end{proof}

We end this section with the computation of $\signk{\gamma} = \|\gnk{\gamma}\|_{L^2(\Dm,d^\gamma)}/ \|\psink{\gamma}\|_{L^2(\partial_+ S\Dm, \mu^{-2\gamma})}$. 

\begin{proof}[Proof of Eq. \eqref{eq:SVD1}] From the previous proof, the family $\{\gnk{\gamma}\}_{n\ge , 0\le k\le n}$ are orthogonal polynomials for $L^2(\Dm, d^\gamma)$, each with the same monomial content as the generalized Zernike polynomial $P_{n-k,k}^\gamma$ defined in \cite{Wuensche2005}. By uniqueness up to scaling of families of orthogonal polynomials, we must have $\gnk{\gamma} = \frac{g_{n,k}^\gamma}{p_{n-k,k}^\gamma} P_{n-k,k}^\gamma$, with $g_{n,k}^\gamma$ defined in \eqref{eq:gnkgamma} and $p_{n-k,k}^\gamma$ defined in \eqref{eq:norms}. Now the singular values are given by 
    \begin{align*}
	(\signk{\gamma})^2 = \frac{\|\gnk{\gamma}\|^2_{L^2(\Dm, d^\gamma)}}{\|\psink{\gamma}\|^2_{L^2(\partial_+ SM, \mu^{-2\gamma}d\Sigma^2)}} &= 2\pi \frac{\|\gnk{\gamma}\|^2_{L^2(\Dm, d^\gamma)}}{\|L_n^\gamma\|^2_{L^2([-1,1], (1-x^2)^{\gamma+1/2})}} \\
	&= \frac{2\pi}{4^n}\binom{n}{k}^2 \frac{\|P_{n-k,k}^\gamma\|^2_{L^2(\Dm, d^\gamma)}}{(p_{n-k,k}^\gamma)^2}\frac{(\ell_n^\gamma)^2}{\|L_n^\gamma\|^2_{L^2([-1,1], (1-x^2)^{\gamma+1/2})}}.
    \end{align*} 
    Notice how this expression is independent of the normalization chosen for each family of orthogonal polynomials. From \eqref{eq:norms} we have
    \begin{align*}
	\frac{\|P_{n-k,k}^\gamma\|^2_{L^2(\Dm, d^\gamma)}}{(p_{n-k,k}^\gamma)^2} = \frac{\pi}{n+\gamma + 1} \frac{(n-k)! k! (n-k+\gamma)! (k+\gamma)!}{ ((n+\gamma)!)^2 }.
    \end{align*}
    From \cite[A.3, A.4]{Wuensche2005}, we also gather
    \begin{align*}
	\|L_n^\gamma\|^2 = \frac{2^{2\gamma+1}}{n+\gamma+1} \frac{n! ((\gamma+1/2)!)^2}{(n+2\gamma+1)!}, \qquad \ell_n^\gamma = 2^n \frac{(2\gamma+1)! (n+\gamma)!}{(n+2\gamma+1)! \gamma!}, 
    \end{align*}
    and hence
    \begin{align*}
	\frac{(\ell_n^\gamma)^2}{\|L_n^\gamma\|^2} = 4^n \frac{n+\gamma+1}{2^{2\gamma+1}} \frac{(2\gamma+1)!^2 (n+\gamma)!^2}{(n+2\gamma+1)! \gamma!^2 n! (\gamma+1/2)!^2}.  
    \end{align*}
    Putting it all together, we obtain
    \begin{align*}
	(\signk{\gamma})^2 &= \frac{\pi^2}{2^{2\gamma}} \left( \frac{\Gamma(2\gamma+2)}{\Gamma(\gamma+1)\Gamma(\gamma+3/2)} \right)^2 \binom{n}{k} \frac{\Gamma(n-k+\gamma+1)\Gamma(k+\gamma+1)}{\Gamma(n+2\gamma+2)}
    \end{align*}
    and expression \eqref{eq:SVD1} follows using Legendre's duplication formula
    \begin{align}
	\Gamma(2z) = \frac{1}{\sqrt{\pi}} 2^{2z-1} \Gamma(z) \Gamma(z+1/2).
	\label{eq:duplication}
    \end{align}
\end{proof}

\section{Isomorphism properties - Proof of Theorem \ref{thm:isomorphism}}\label{sec:isomorphism}

\subsection{Outline of the proof} \label{sec:outline2}

For $\gamma>-1$ fixed, the main idea of the proof is to construct a scale of Sobolev-type spaces $(\wtH^{s,\gamma}(\Dm), \|\cdot\|_{s,\gamma})_{s\ge 0}$ which intersects to $C^\infty(\Dm)$, on which one may obtain two-sided continuity and stability estimates with uniform regularity gain/loss. In other words, the Fr\'echet topology of $C^\infty(\Dm)$ can be described in terms of the countable family of semi-norms $\{\|\cdot\|_{s,\gamma}\}_{s\in \Nm_0}$, graded in the sense that $\|\cdot\|_{s+1,\gamma} \ge \|\cdot\|_{s,\gamma}$ for all $s$. In this context, we show that $I_0^\sharp \mu^{-2\gamma-1} I_0 d^\gamma$ is tame (in the sense that there exists a fixed $\ell$ in dependent of $s$ such that $I_0^\sharp \mu^{-2\gamma-1} I_0 d^\gamma (\wtH^{s,\gamma}) \subset \wtH^{s+\ell,\gamma}$ for all $s$) and with tame inverse. 

On to the outline of the proof, we first give $k$-independent, large-$n$ asymptotics of the singular values of $I_0^\sharp \mu^{-2\gamma-1} I_0 d^\gamma$.

\begin{lemma}\label{lem:SVDasym}
    There exist constants $C_1$ and $C_2$ (depending on $\gamma$) such that
    \begin{align*}
	C_1(n+1)^{\min(-1,-1-\gamma)}\le (\signk{\gamma})^2\le C_2(n+1)^{\max(-1,-1-\gamma)}, \quad n\ge 0,\quad 0\le k\le n,
    \end{align*}
    where $\signk{\gamma}$ is defined in \eqref{eq:SVD2}.
\end{lemma}

For any $\gamma>-1$ and $s \ge 0$, let us define $\wtH^{s,\gamma}(\Dm) := \D(\L_\gamma^{s/2})$, where $\L_\gamma$ is the self-adjoint operator from Section \ref{sec:Lgamma}. If $\widehat{\gnk{\gamma}}$ denotes the $L^2(\Dm,d^{\gamma}\,dx)$-normalized generalized Zernike polynomials, by virtue of \eqref{eq:eigenequations}, we also have
\begin{align}
    \wtH^{s,\gamma}(\Dm) = \left\{f = \sum_{\substack{n\ge 0\\ 0\le k\le n}}{f_{n,k}\ \widehat{\gnk{\gamma}}}\,\mid\,\|f\|_{\wtH^{s,\gamma}(\Dm)}^2:=\sum_{\substack{n\ge 0\\ 0\le k\le n}}{(n+1+\gamma)^{2s} |f_{n,k}|^2}<\infty\right\}.
    \label{eq:Sobolev}
\end{align}
In particular, $\wtH^{0,\gamma}(\Dm) = L^2(\Dm, d^{\gamma})$, and $\|\L_{\gamma}f\|_{\wtH^{s,\gamma}(\Dm)} = \|f\|_{\wtH^{s+2,\gamma}(\Dm)}$ for all $s\ge 0$ and $f\in \wtH^{s+2,\gamma}(\Dm)$. We first establish the following fact about the generalized Zernike polynomials. 
\begin{lemma}\label{lem:linftybound}
    There exists a constant $C$ (depending on $\gamma$, but not on $n$ or $k$) such that
    \begin{align*}
	\frac{\|\gnk{\gamma}\|_{L^{\infty}}}{\|\gnk{\gamma}\|_{L^2(\Dm, d^{\gamma}\,dx)}}\le C \left\{
	\begin{array}{cc}
	    (n+1)^{1+\gamma} & (\gamma \le 0), \\
	    (n+1)^{1+\frac{3}{2}\gamma} & (\gamma>0). 	    
	\end{array}
	\right.
    \end{align*}
\end{lemma}

\begin{remark}
    In the case $\gamma=0$, this result is asymptotically $n^{1/2}$ weaker than the optimal result where for the Zernike polynomials we have $\|Z_{n,k}\|_{L^{\infty}}/\|Z_{n,k}\|_{L^2} = \pi^{-1/2}(n+1)^{1/2}$; see \cite{Monard2017}.
\end{remark}

The next lemma makes use of properties enjoyed by the generalized Zernike polynomials as a triply-indexed family (over $\gamma, n, k$). Specifically, the results in \cite{Wuensche2005} help show that the action of $\frac{\partial}{\partial z}$ and $\frac{\partial}{\partial \zbar}$ is sharply understood between two bases with exponents $\gamma$ differing by $1$.

\begin{lemma}\label{lem:diffmap}
    For any $s\ge 0$ and $\gamma>-1$, the operators $\frac{\partial}{\partial z}$ and $\frac{\partial}{\partial \zbar}$ are $\wtH^{s+1,\gamma}(\Dm)\to \wtH^{s,\gamma+1}(\Dm)$-bounded.
\end{lemma}

Lemmas \ref{lem:linftybound} and \ref{lem:diffmap} then help us prove the following inclusions. 
\begin{lemma}\label{lem:continc}
    (a) We have a continuous inclusion $\wtH^{s,\gamma}(\Dm)\hookrightarrow C(\Dm)$ whenever $s>2+\max(\gamma,\frac{3}{2}\gamma)$.

    (b) For $k\ge 1$, we have a continuous inclusion $\wtH^{s,\gamma}(\Dm)\hookrightarrow C^k(\Dm)$ whenever $s>{2+\frac{3}{2}\gamma+\frac{5}{2}k}$.

    (c) For any $\gamma>-1$, $\cap_s \wtH^{s,\gamma}(\Dm) = C^{\infty}(\Dm)$.
\end{lemma}

Concluding with the mapping properties of $I_0^\sharp \mu^{-2\gamma-1} I_0 d^\gamma$, the asymptotics from Lemma \ref{lem:SVDasym} immediately imply the estimates 
\begin{align*}
    C_1 \|f\|_{\wtH^{s+\min(-1,-1-\gamma),\gamma}} \le \|I_0^\sharp \mu^{-2\gamma-1} I_0 d^\gamma f\|_{\wtH^{s,\gamma}} \le C_2 \|f\|_{\wtH^{s+\max(-1,-1-\gamma),\gamma}},
\end{align*}
for all $s$ such that all three Sobolev exponents are non-negative. It follows that 
\begin{align*}
    \wtH^{s+\max(1,1+\gamma),\gamma}(\Dm)  \subset I_0^\sharp \mu^{-2\gamma-1} I_0 d^\gamma (\wtH^{s,\gamma}(\Dm)) \subset \wtH^{s+\min(1,1+\gamma),\gamma}(\Dm), \quad s\ge 0,
\end{align*}
where the left inclusion follows from the open mapping theorem. Combining this with Lemma \ref{lem:continc}.(c), the $C^\infty(\Dm)$-isomorphism property follows, and Theorem \ref{thm:isomorphism} is proved. 

\subsection{Proofs of remaining lemmas}\label{sec:lemmas2}
\paragraph{Singular Value asymptotics.}

\begin{proof}[Proof of Lemma \ref{lem:SVDasym}] From expression \eqref{eq:SVD1}, we have
    \begin{align*}
	(\signk{\gamma})^2 &= 2^{2\gamma+2}\pi\binom{n}{k}\frac{\Gamma(n-k+\gamma+1)\Gamma(k+\gamma+1)}{\Gamma(n+2\gamma+2)} \\
	&= 2^{2\gamma+2}\pi\frac{\Gamma(n+1)}{\Gamma((n+1)+2\gamma+1)}\frac{\Gamma(n-k+1+\gamma)}{\Gamma(n-k+1)}\frac{\Gamma(k+1+\gamma)}{\Gamma(k+1)}.
    \end{align*}
    To obtain $k$-independent asymptotics, a first observation is that $\sigma_{n,n-k}^{\gamma} = \signk{\gamma}$, so that it suffices to study the case $0\le k\le n/2$. Second, from the asymptotic formula $\frac{\Gamma(x+a)}{\Gamma(x)}\sim x^a$ as $x\to \infty$, we deduce the specific two asymptotic regimes 
    \begin{align*}
	(\sigma_{n,0}^{\gamma})^2\sim 2^{2\gamma+2}\pi\Gamma(1+\gamma)n^{-1-\gamma} \qquad (\sigma_{n,\lfloor{n/2}\rfloor}^{\gamma})^2\sim 4\pi n^{-1}, \qquad n\to \infty.
    \end{align*}
    Note that for $\gamma<0$, $\sigma_{n,0}^{\gamma}$ is asymptotically\footnote{in fact, as explained in the next line, the inequality holds always and not just asymptotically} larger than $\sigma_{n,\lfloor{n/2}\rfloor}^{\gamma}$, while for $\gamma>0$ the reverse is true.
    
    Finally, we claim that all other coefficients fall between $\sigma_{n,0}^{\gamma}$ and $\sigma_{n, \lfloor n/2 \rfloor}^{\gamma}$. More precisely we show
    \begin{lemma}\label{lem:cnkgamma}
	Let $c_{n,k}^{\gamma} = \sigma_{n,k+1}^{\gamma}/\sigma_{n,k}^{\gamma}$. For $\gamma<0$, we have $c_{n,k}^{\gamma}<1$ if $k<(n-1)/2$ and $c_{n,k}^{\gamma}>1$ if $k>(n-1)/2$. For $\gamma>0$, we have $c_{n,k}^{\gamma}>1$ if $k<(n-1)/2$ and $c_{n,k}^{\gamma}<1$ if $k>(n-1)/2$.
    \end{lemma}
    As a consequence, we see that for a fixed $n$ we have that $\sigma_{n,k}$ attains its minimum value at $k=\lfloor n/2\rfloor$ and its maximum value at $k=0$ and $k=n$ if $\gamma<0$, while if $\gamma>0$ then the reverse is true. 
    
    \begin{proof}[Proof of Lemma \ref{lem:cnkgamma}]
	From the property $\Gamma(z+1)=z\Gamma(z)$, we see that
	\begin{equation}
	    \label{svratio}
	    (c_{n,k}^{\gamma})^2 = \frac{n-k}{n-k+\gamma}\cdot\frac{k+1+\gamma}{k+1} = \left(1+\frac{\gamma}{k+1}\right)/\left(1+\frac{\gamma}{n-k}\right).
	\end{equation}
	If $\gamma<0$, then for $a,b>0$ we have $1+\frac{\gamma}{a}>1+\frac{\gamma}{b}\iff a>b$, and hence $1+\frac{\gamma}{k+1}<1+\frac{\gamma}{n-k}\iff k+1<n-k\iff k<(n-1)/2$, so $c_{n,k}^{\gamma}<1\iff k<(n-1)/2$ when $\gamma<0$. Analogously, $c_{n,k}^{\gamma}>1\iff k>(n-1)/2$ when $\gamma<0$. The analogous statements for $\gamma>0$ also follow from the above observation, noting now that for $a,b>0$ we have $1+\frac{\gamma}{a}>1+\frac{\gamma}{b}\iff a<b$. 
    \end{proof}

    From this, we have $(\sigma_{n,k}^{\gamma})^2\le\max((\sigma_{n,0}^{\gamma})^2,(\sigma_{n,\lfloor n/2\rfloor}^{\gamma})^2)\le C\max(n^{-1-\gamma},n^{-1})$, thus proving Lemma \ref{lem:SVDasym}.
        
\end{proof}

\paragraph{$L^\infty$-estimate for $\widehat{\gnk{\gamma}}$.}

\begin{proof}[Proof of Lemma \ref{lem:linftybound}]
    To obtain an estimate for $\|\gnk{\gamma}\|_{L^{\infty}}/\|\gnk{\gamma}\|_{L^2(\Dm, d^{\gamma})}$, our starting point is the expression \eqref{eq:lntognk} of $\gnk{\gamma}$ in terms of $L_n^{\gamma}(x)$, which immediately implies the crude bound
    \begin{align}
	\|\gnk{\gamma}\|_{L^{\infty}(\Dm)}\le \|L_n^{\gamma}\|_{L^{\infty}([-1,1])}.
	\label{eq:bound1}
    \end{align}
    Here, $\{L_n^\gamma\}_n$ are orthogonal in $L^2([-1,1],(1-x^2)^{\gamma+1/2}\,dx)$, thus they are multiples of the Jacobi polynomials $P^{(\gamma+1/2,\gamma+1/2)}_n$, or alternatively multiples of the ultraspherical/Gegenbauer polynomials $C^{(\gamma+1)}_n$; see \cite{Szegoe1938}. In particular, 
    \begin{align*}
	\frac{\|L_n^{\gamma}\|_{L^{\infty}([-1,1])}}{|l_n^\gamma|} = \frac{\|P^{(\gamma+1/2,\gamma+1/2)}_n\|_{L^{\infty}}}{|p^{(\gamma+1/2,\gamma+1/2)}_n|} = \frac{2^n(n+\gamma+1/2)!(n+2\gamma+1)!}{(\gamma+1/2)!(2n+2\gamma+1)!},
    \end{align*}
    where the lower-case letters indicate leading coefficients of the corresponding polynomial, and where the last equality comes from using the following formulas, see \cite{Szegoe1938}
    \begin{align*}
	\|P^{(\gamma+1/2,\gamma+1/2)}_n\|_{L^{\infty}} = \binom{n+\gamma+1/2}{n} = \frac{(n+\gamma+1/2)!}{n!(\gamma+1/2)!}, \qquad |p^{(\gamma+1/2,\gamma+1/2)}_n| = \frac{(2n+2\gamma+1)!}{2^nn!(n+2\gamma+1)!}. 
    \end{align*}
    From Legendre's duplication formula \eqref{eq:duplication}, it follows that
    \begin{align*}
	(2n+2\gamma+1)! = \Gamma(2n+2\gamma+2) &= \frac{1}{\sqrt{\pi}}2^{2n+2\gamma+1}\Gamma(n+\gamma+1)\Gamma(n+\gamma+3/2) \\
	&= \frac{1}{\sqrt{\pi}}2^{2n+2\gamma+1}(n+\gamma)!(n+\gamma+1/2)!    
    \end{align*}
    and hence
    \begin{align}
	\frac{\|L_n^{\gamma}\|_{L^{\infty}}}{|l_n^{\gamma}|} = \frac{\sqrt{\pi}(n+2\gamma+1)!}{2^{n+2\gamma+1}(\gamma+1/2)!(n+\gamma)!}.
	\label{eq:Lratio}
    \end{align}
    Combining \eqref{eq:bound1}, \eqref{eq:Lratio} and \eqref{eq:gnkgamma}, we arrive at 
    \begin{align}
	\|\gnk{\gamma}\|_{L^{\infty}} \le \frac{\sqrt{\pi}(n+2\gamma+1)!}{2^{n+2\gamma+1}(\gamma+1/2)!(n+\gamma)!}|l_n^{\gamma}| = \frac{\sqrt{\pi}(n+2\gamma+1)!k!(n-k)!}{2^{2\gamma+1}(\gamma+1/2)!n!(n+\gamma)!}|g_{n,k}^{\gamma}|.
	\label{eq:bound2}
    \end{align}
    On to getting a handle on the $L^2$ norm, $\gnk{\gamma}$ are multiples of the generalized Zernike polynomials $P^{\gamma}_{n-k,k}$ in \cite{Wuensche2005}, and using \eqref{eq:norms}, we have
    \begin{align}
	\frac{\|G_{n,k}^{\gamma}\|_{L^2(\Dm, d^{\gamma})}^2}{|g_{n,k}^{\gamma}|^2} = \frac{\|P^{\gamma}_{n-k,k}\|_{L^2(\Dm, d^{\gamma})}^2}{|p^{\gamma}_{n-k,k}|^2} \stackrel{\eqref{eq:norms}}{=} \frac{\pi (n-k)!k!(n-k+\gamma)!(k+\gamma)!}{(n+\gamma+1)((n+\gamma)!)^2}.
	\label{eq:L2}
    \end{align}
    Combining \eqref{eq:bound2} with \eqref{eq:L2}, 
    \begin{align*}
	\|\gnk{\gamma}\|_{L^{\infty}}^2 &\le \frac{\pi((n+2\gamma+1)!k!(n-k)!)^2}{2^{4\gamma+2}((\gamma+1/2)!n!(n+\gamma)!)^2}|g_{n,k}^{\gamma}|^2 \\
	&=\frac{\pi((n+2\gamma+1)!k!(n-k)!)^2}{2^{4\gamma+2}((\gamma+1/2)!n!(n+\gamma)!)^2}\frac{(n+\gamma+1)((n+\gamma)!)^2}{\pi (n-k)!k!(n-k+\gamma)!(k+\gamma)!}\|G_{n,k}^{\gamma}\|_{L^2(\Dm,d^{\gamma})}^2 \\
	&=\frac{k!(n-k)!(n+\gamma+1)((n+2\gamma+1)!)^2}{2^{4\gamma+2}((\gamma+1/2)!n!)^2(n-k+\gamma)!(k+\gamma)!}\|G_{n,k}^{\gamma}\|_{L^2(\Dm,d^{\gamma})}^2
    \end{align*}
    Let $C_{n,k}^2 := \frac{k!(n-k)!(n+\gamma+1)((n+2\gamma+1)!)^2}{2^{4\gamma+2}((\gamma+1/2)!n!)^2(n-k+\gamma)!(k+\gamma)!}$ be the constant appearing above. Noting that $C_{n,n-k}^2 = C_{n,k}^2$, it is enough to study the case $0\le k\le n/2$. Similarly to the proof of Lemma \ref{lem:SVDasym}, we show that obtaining $k$-independent asymptotics for $C_{n,k}^2$ involves showing that all terms are between the following two asymptotic regimes, as $n\to \infty$,
    \begin{align*}
	C_{n,0}^2\sim\frac{1}{(2^{2\gamma+1}(\gamma+1/2)!)^2(\gamma)!}n^{3\gamma+2} \quad\text{and}\quad C_{n,\lfloor n/2\rfloor}^2\sim\frac{1}{(2^{\gamma+1}(\gamma+1/2)!)^2}n^{2\gamma+2}. 
    \end{align*}
    Indeed, letting $c_{n,k} := C_{n,k+1}/C_{n,k}$, we have
    \[c_{n,k}^2 = \frac{(k+1)(n-k+\gamma)}{(k+1+\gamma)(n-k)}.\]
    This is exactly the inverse of the quantity in \eqref{svratio}, and hence this is greater than $1$ for $k<(n-1)/2$ and less than $1$ for $k>(n-1)/2$ if $\gamma<0$; an analogous statement holds for $\gamma>0$. As a result, if $\gamma<0$, then $C_{n,k}$ is maximized (for fixed $n$) at $k=\lfloor n/2\rfloor$ and minimized at $k\in \{0,n\}$ (hence the estimate ${\cal O} ((n+1)^{\gamma+1})$), while if $\gamma>0$ the opposite is true (hence the estimate ${\cal O} ( (n+1)^{\frac{3}{2}\gamma+1})$). Lemma \ref{lem:linftybound} is proved.      
\end{proof}

\paragraph{Properties of the Sobolev scale $\wtH^{s,\gamma}$.}

\begin{proof}[Proof of Lemma \ref{lem:diffmap}]
    From \cite{Wuensche2005}, we have, for all $m,\ell\in \Nm_0$ and $\gamma>-1$, 
    \begin{align*}
	\frac{\partial}{\partial z}P_{m,\ell}^{\gamma}(z) = \frac{m(\ell+1+\gamma)}{1+\gamma}P_{m-1,\ell}^{\gamma+1}(z) \qquad \text{and}\qquad \frac{\partial}{\partial\overline{z}}P_{m,\ell}^{\gamma}(z) = \frac{(m+1+\gamma)\ell}{1+\gamma}P_{m,\ell-1}^{\gamma+1}(z).    
    \end{align*}
    Since
    \begin{align*}
	\|P_{m,\ell}^{\gamma}\|_{L^2(\Dm,d^{\gamma})}^2 = \frac{\pi m!(\gamma !)^2 \ell!}{(m+\ell+\gamma+1)(m+\gamma)!(\ell+\gamma)!},	
    \end{align*}
    it follows that for the $L^2(\Dm,d^{\gamma})$-normalized polynomials $\widehat{P_{m,\ell}^{\gamma}}$ we have
    \[\frac{\partial}{\partial z}\widehat{P_{m,\ell}^{\gamma}} = c_{m,\ell,\gamma}\ \widehat{P_{m-1,\ell}^{\gamma+1}}\]
    where $c_{m,\ell,\gamma} = \frac{m(\ell+1+\gamma)}{1+\gamma}\cdot\frac{\|P_{m-1,\ell}^{\gamma+1}\|_{L^2(\Dm, d^{\gamma+1})}}{\|P_{m,\ell}^{\gamma}\|_{L^2(\Dm, d^{\gamma})}}$; hence
    \[c_{m,\ell,\gamma}^2 = \frac{m^2(\ell+1+\gamma)^2(m-1)!((\gamma+1)!)^2(\ell+\gamma)!}{(1+\gamma)^2m!(\gamma!)^2(\ell+\gamma+1)!} = \frac{m^2(\ell+1+\gamma)^2(\gamma+1)^2}{(\gamma+1)m(\ell+\gamma+1)} = (\gamma+1)m(\ell+\gamma+1).\]
    Similarly, for all $m,\ell\in \Nm_0$ and $\gamma>-1$, 
    \begin{align*}
	\frac{\partial}{\partial\overline{z}}\widehat{P_{m,\ell}^{\gamma}} = c_{m,\ell,\gamma}^*\ \widehat{P_{m,\ell-1}^{\gamma+1}}, \quad\text{where}\quad (c_{m,\ell,\gamma}^*)^2 := (\gamma+1)(m+\gamma+1)\ell.    
    \end{align*}
    It follows that, with $\widehat{\gnk{\gamma}} = (-1)^k \widehat{P_{n-k,k}^\gamma}$,
    \begin{align*}
	\frac{\partial}{\partial z}\widehat{\gnk{\gamma}} &= ((\gamma+1)(n-k)(k+\gamma+1))^{1/2}\ \widehat{Z_{n-1,k}^{\gamma+1}} \\
	\frac{\partial}{\partial\overline{z}}\widehat{\gnk{\gamma}} &= -((\gamma+1)(n-k+\gamma+1)k)^{1/2}\ \widehat{Z_{n-1,k-1}^{\gamma+1}}.	
    \end{align*}
    (For convention purposes, we take $\widehat{\gnk{\gamma}} = 0$ if $k<0$ or $n-k<0$.) Thus, for $f = \sum{f_{n,k}\widehat{\gnk{\gamma}}}$, we have
    \begin{align*}
	\left\|\frac{\partial}{\partial z}f\right\|_{\wtH^{s,\gamma+1}(\Dm)}^2 &= \left\|\sum_{n\ge 0, 0\le k\le n}{((\gamma+1)(n-k)(k+\gamma+1))^{1/2}f_{n,k}\ \widehat{Z_{n-1,k}^{\gamma+1}}}\right\|_{\wtH^{s,\gamma+1}(\Dm)}^2 \\
	&=\left\|\sum_{n\ge 0, 0\le k\le n}{((\gamma+1)(n+1-k)(k+\gamma+1))^{1/2}f_{n+1,k}\ \widehat{\gnk{\gamma+1}}}\right\|_{\wtH^{s,\gamma+1}(\Dm)}^2 \\
	&=(\gamma+1)\sum_{n\ge 0, 0\le k\le n}{(n+1-k)(k+\gamma+1)(n+(\gamma+1)+1)^{2s}|f_{n+1,k}|^2} \\
	&\le(\gamma+1)\sum_{n\ge 0, 0\le k\le n}{\frac{(n+\gamma+2)^2}{4}(n+\gamma+2)^{2s}|f_{n+1,k}|^2} \\
	&=\frac{\gamma+1}{4}\sum_{n\ge 0, 0\le k\le n}{(n+\gamma+2)^{2s+2}|f_{n+1,k}|^2} \\
	&=\frac{\gamma+1}{4}\sum_{n\ge 1, 0\le k\le n-1}{(n+\gamma+1)^{2(s+1)}|f_{n,k}|^2} \le\frac{\gamma+1}{4}\|f\|_{\wtH^{s+1,\gamma}(\Dm)}^2.
    \end{align*}
    (In the fourth line, we used the fact that $(n+1-k)(k+\gamma+1)\le (n+\gamma+2)^2/4$ for any value of $k$.) Similarly, there is a constant $C$ such that
    \begin{align*}
	\left\|\frac{\partial}{\partial\overline{z}}f\right\|_{\wtH^{s,\gamma+1}(\Dm)}^2\le  C\|f\|_{\wtH^{s+1,\gamma}(\Dm)}^2,  
    \end{align*}
    and Lemma \ref{lem:diffmap} is proved.     
\end{proof}

\begin{proof}[Proof of Lemma \ref{lem:continc}]
    (Proof of (a)) By the Weierstrass M-test, it suffices to check that the series 
    \begin{align*}
	\sum_{n\ge 0}\sum_{0\le k\le n}{|f_{n,k}|\left\|\widehat{\gnk{\gamma}}\right\|_{L^{\infty}}}
    \end{align*}
    converges when $f = \sum{f_{n,k}\widehat{\gnk{\gamma}}}\in \wtH^{s,\gamma}(\Dm)$. From Cauchy-Schwarz, we have
    \begin{align*}
	\sum_{n\ge 0}\sum_{0\le k\le n}{|f_{n,k}| \left\|\widehat{\gnk{\gamma}}\right\|_{L^{\infty}}}\le\left(\sum_{n\ge 0}\sum_{0\le k\le n}{\frac{\left\|\widehat{\gnk{\gamma}}\right\|_{L^{\infty}}^2}{(n+1+\gamma)^{2s}}}\right)^{1/2}\|f\|_{\wtH^{s,\gamma}(\Dm)}.	
    \end{align*}
    and
    \begin{align*}
	\sum_{n\ge 0}\sum_{0\le k\le n}{\frac{\left\|\widehat{\gnk{\gamma}}\right\|_{L^{\infty}}^2}{(n+1+\gamma)^{2s}}}\le\sum_{n\ge 0}\sum_{0\le k\le n}{\frac{C(n+1)^{2+\max(2\gamma,3\gamma)}}{(n+1+\gamma)^{2s}}} = C\sum_{n\ge 0}{\frac{(n+1)^{3+\max(2\gamma,3\gamma)}}{(n+1+\gamma)^{2s}}}.	
    \end{align*}
    The summand is asymptotic to $n^{3+\max(2\gamma,3\gamma)-2s}$, so the sum converges when $3+\max(2\gamma,3\gamma)-2s<-1$, i.e. when $s>2+\max(\gamma,\frac{3}{2}\gamma)$.

    (Proof of (b)) By Lemma \ref{lem:diffmap}, we see that for \emph{any} multi-index $\alpha$, $\partial^{\alpha}$ maps continuously from $\wtH^{s+|\alpha|,\gamma}$ to $\wtH^{s,\gamma+|\alpha|}$, or equivalently from $\wtH^{s,\gamma}$ to $\wtH^{s-|\alpha|,\gamma+|\alpha|}$ for $s\ge|\alpha|$. Combining this with (a), for $\gamma>-1$ and $|\alpha|\ge 1$, we have $\partial^{\alpha}f\in C(\Dm)$ if $f\in \wtH^{s,\gamma}$ with $s-|\alpha|>2+\frac{3}{2}(\gamma+|\alpha|)$, i.e. if $s>2+\frac{3}{2}\gamma+\frac{5}{2}|\alpha|$.

    (Proof of (c)) The inclusion $(\supset)$ is clear since $\|f\|_{\wtH^{2n,\gamma}}(\Dm) = \|(\L_{\gamma})^n f\|_{L^2(\Dm,d^{\gamma}\,dx)}$ will be finite for all $n$ whenever $f\in C^\infty(\Dm)$. On to proving $(\subset)$, suppose $f\in \cap_s \wtH^{s,\gamma}(\Dm)$ and fix $p,q\in \Nm_0$. For all $s \ge 0$, since $f\in \wtH^{s+p+q,\gamma}(\Dm)$, Lemma \ref{lem:diffmap} implies that $\partial_z^p \partial_{\zbar}^q f \in \wtH^{s,\gamma+p+q}(\Dm)$, then using (a) with any $s>2+\max(\gamma+p+q, \frac{3}{2} (\gamma+p+q))$, we obtain that $\partial_z^p \partial_{\zbar}^q f\in C(\Dm)$, and thus $f\in C^\infty(\Dm)$.
\end{proof}

\subsection{Proof of Theorem \ref{thm:rangeCharac}} \label{sec:proofrangeCharac}

We now provide a proof of the range characterization. 

\begin{proof}[Proof of Theorem \ref{thm:rangeCharac}]
    A function $\psi\in \mu^{2\gamma+1}C_{\alpha,+,+}^\infty(\partial_+ S\Dm)$ can be written as $\psi = \sum_{n\ge 0, k\in \Zm} a_{n,k} \widehat{\psink{\gamma}}$ where $(\T_\gamma)^p \partial_\beta^q \psi \in \mu^{2\gamma+1}C_{\alpha,+,+}^\infty(\partial_+ S\Dm)$ for all $p,q\in \Nm_0$. This implies the condition
    \begin{align}
	\sum_{n\ge 0,\ k\in \Zm} (n+1+\gamma)^{4p} (n-2k)^{2q} |a_{n,k}|^2 <\infty, \quad p,q\in \Nm_0. 
	\label{eq:decay}
    \end{align}
    Now from Theorem \ref{thm:SVD}, such a function $\psi$ is equal to $I_0 d^\gamma f$, where $f := \sum_{n\ge 0, 0\le k\le n} \frac{a_{n,k}}{\sigma_{n,k}^\gamma} \widehat{\gnk{\gamma}}$, if and only if $a_{n,k} = 0$ for $k\notin \{0, 1, \dots, n\}$, which is equivalent to \eqref{eq:HLGG}. In the case where such an equality holds, the fact that $f\in C^\infty(\Dm)$ comes from combining \eqref{eq:decay} (with $p\in \Nm$ and $q=0$), Lemma \ref{lem:SVDasym} and Lemma \ref{lem:continc}.(c), as $f\in \cap_{p\ge 0} \wtH^{p,\gamma}(\Dm) = C^\infty(\Dm)$. Theorem \ref{thm:rangeCharac} is proved.
\end{proof}

\section{Constant curvature disks - Proof of Theorem \ref{thm:CCD}}\label{sec:CCDs}

Following notation in \cite{Monard2019a}, fix $\kappa\in \Rm$ and $R>0$ such that $R^2|\kappa|<1$, define $\Dm_R = \{(x,y)\in \Rm^2: x^2+y^2\le R^2\}$, equipped with the metric $g_\kappa(z) := (1+\kappa|z|^2)^{-2} |dz|^2$, of constant curvature $4\kappa$. We denote $S_{(\kappa)}\Dm_R$ the unit tangent bundle
\begin{align*}
    S_{(\kappa)}\Dm_R = \{(x,v) \in T\Dm_R,\quad (g_\kappa)_x(v,v) = 1\},
\end{align*}
with inward boundary $\partial_+ S_{(\kappa)}\Dm_R$ defined as usual. The latter is parameterized in fan-beam coordinates $(\beta,\alpha)\in \Sm^1\times (-\pi/2,\pi/2)$, where $\beta$ describes the boundary point $x = R e^{i\beta}$, and $\alpha$ describes the angle of the tangent vector with respect to the unit inner normal $\nu_x$, i.e. $v = (1+R^2\kappa) e^{i(\beta+\pi+\alpha)}$. The manifold $\partial_+ S_{(\kappa)}\Dm_R$ is a model for all geodesics on $\Dm_R$ intersecting $\partial \Dm_R$ transversally, equipped with the measure $d\Sigma^2 = R (1+R^2\kappa)^{-1} d\beta\ d\alpha$.

Keeping $R, \kappa$ fixed, we introduce the diffeomorphism $\Phi\colon \Dm_R\to \Dm_1$ as $\Phi(z) := \frac{1-\kappa R^2}{1-\kappa |z|^2} \frac{z}{R}$, as well as the function $w(z) := \frac{1+\kappa |z|^2}{1-\kappa |z|^2}$. The following change of volume formula holds, as proved in \cite[Eq. (51)]{Monard2019a}
\begin{align}
    \Phi^*(d\Vol_e) = \frac{(1-\kappa R^2)^2}{R^2} w^3 \ d\Vol_\kappa
    \label{eq:chgvol}
\end{align}

With $d_e(z) := 1-|z|^2$ for $z\in \Dm$, a first observation is that $d := d_e \circ \Phi$ is a boundary defining function\footnote{in the sense that it is smooth on $\Dm_R$, positive in the interior, and vanishes to first order at $\partial \Dm_R$} for $\Dm_R$. Indeed, one computes directly that 
\begin{align*}
    d(z):= d_e (\Phi(z)) = \left( 1-\frac{|z|^2}{R^2} \right) \frac{1-\kappa^2 R^2 |z|^2}{1-\kappa |z|^2}, \qquad z\in \Dm_R, 
\end{align*}
where, under the constraint $|\kappa|R^2<1$, the second factor is non-vanishing on $\Dm_R$.

Define the map $\ss\colon \partial_+ S_{(\kappa)} \Dm_R \to \partial_+ S\Dm$ by  
\begin{align*}
    \ss(\beta,\alpha) := \left(\beta, \tan^{-1}\left( \frac{1-R^2\kappa}{1+R^2\kappa} \tan\alpha \right)\right),
\end{align*}
a diffeomorphism with Jacobian determinant denoted $\ss'$, in particular such that 
\begin{align}
    \ss^* (d\Sigma_e^2) = \frac{1+\kappa R^2}{R} \ss'\ d\Sigma^2.
    \label{eq:chgvol2}
\end{align}
It is further proved in \cite[Eq. (44)]{Monard2019a} that 
\begin{align}
    \mu_e \circ \ss = \sqrt{\frac{1-\kappa R^2}{1+\kappa R^2}} \sqrt{\ss'} \mu \qquad \text{ on } \qquad \partial_+ S_{(\kappa)}\Dm_R.
    \label{eq:murel}
\end{align}

The starting point is the intertwining relation given in \cite[Theorem 18]{Monard2019a}, with adjusted notation:  
\begin{align}
    (I_0^\sharp \mu^{-1})^e = \sqrt{\frac{1+\kappa R^2}{1-\kappa R^2}} (\Phi^{-1})^* \frac{1}{w} (I_0^\sharp \mu^{-1}) \sqrt{\ss'} \ss^*.
    \label{eq:interIstar}
\end{align}
It can be viewed as the equality between two operators on $C^\infty( (\partial_+ S\Dm)^{int} )$. We now compute
\begin{align*}
    (I_0^\sharp \mu^{-1-2\gamma})^e = (I_0^\sharp \mu^{-1})^e \mu_e^{-2\gamma}\ &\!\! \stackrel{\eqref{eq:murel}}{=} \left( \frac{1+\kappa R^2}{1-\kappa R^2} \right)^{\gamma} (I_0^\sharp \mu^{-1})^e ( (\sqrt{\ss'} \mu)^{-2\gamma} )\circ \ss^{-1} \\
    &\!\! \stackrel{\eqref{eq:interIstar}}{=} \left( \frac{1+\kappa R^2}{1-\kappa R^2} \right)^{\gamma+1/2} (\Phi^{-1})^* \frac{1}{w} (I_0^\sharp \mu^{-1}) \sqrt{\ss'} \ss^* ( (\sqrt{\ss'} \mu)^{-2\gamma} )\circ \ss^{-1} \\
    &= \left( \frac{1+\kappa R^2}{1-\kappa R^2} \right)^{\gamma+1/2} (\Phi^{-1})^* \frac{1}{w} (I_0^\sharp \mu^{-1-2\gamma}) (\sqrt{\ss'})^{-2\gamma+1} \ss^*, 
\end{align*}
which is an intertwining relation between $(I_0^\sharp \mu^{-1-2\gamma})^e$ and $I_0^\sharp \mu^{-1-2\gamma}$. It is an equality between two operators on $C^\infty((\partial_+ S\Dm)^{int})$, in particular as operators on $\mu^{1+2\gamma} C_{\alpha,+,+}^\infty(\partial_+ S\Dm)$, and by density and boundedness, on $L^2_+(\partial_+ S\Dm, \mu_e^{-2\gamma})$. We can then work out the corresponding intertwining relation for the adjoints. Namely, for $f\in L^2(\Dm, d^\gamma)$ and $g\in L^2_+ (\partial_+ S\Dm, \mu_e^{-2\gamma})$, 
\begin{align*}
    \langle g, I_0^e (d^\gamma f) \rangle_{\mu_e^{-2\gamma} d\Sigma_e^2} &= \langle (I_0^\sharp \mu^{-1-2\gamma})^e g, f \rangle_{d^\gamma d\Vol_e} \\
    &= \left( \frac{1+\kappa R^2}{1-\kappa R^2} \right)^{\gamma+1/2} \left\langle (\Phi^{-1})^* \frac{1}{w} (I_0^\sharp \mu^{-1-2\gamma}) (\sqrt{\ss'})^{-2\gamma+1} \ss^* g, f \right\rangle_{d_e^\gamma d\Vol_e} \\
    &= \left( \frac{1+\kappa R^2}{1-\kappa R^2} \right)^{\gamma+1/2} \left\langle \frac{1}{w} (I_0^\sharp \mu^{-1-2\gamma}) (\sqrt{\ss'})^{-2\gamma+1} \ss^* g, \Phi^* f \right\rangle_{(\Phi^* d_e)^\gamma \Phi^* (d\Vol_e)} \\
    &\!\!\stackrel{\eqref{eq:chgvol}}{=} \frac{1}{R^2} \frac{(1+\kappa R^2)^{\gamma+1/2}}{(1-\kappa R^2)^{\gamma-3/2}}  \left\langle (I_0^\sharp \mu^{-1-2\gamma}) (\sqrt{\ss'})^{-2\gamma+1} \ss^* g, w^2 \Phi^* f \right\rangle_{(\Phi^* d_e)^\gamma d\Vol_\kappa} \\
    &= \frac{1}{R^2} \frac{(1+\kappa R^2)^{\gamma+1/2}}{(1-\kappa R^2)^{\gamma-3/2}} \left\langle \ss^* g, I_0 (\Phi^* d_e)^\gamma w^2 \Phi^* f \right\rangle_{ (\sqrt{\ss'})^{-2\gamma+1} \mu^{-2\gamma} d\Sigma^2} \\
    &\!\!\stackrel{\eqref{eq:chgvol2}}{=} \frac{1}{R} \frac{(1+\kappa R^2)^{\gamma-1/2}}{(1-\kappa R^2)^{\gamma-3/2}} \left\langle \ss^* g, I_0 (\Phi^* d_e)^\gamma w^2 \Phi^* f \right\rangle_{ (\sqrt{\ss'})^{-2\gamma-1} \mu^{-2\gamma} \ss^* (d\Sigma_e^2)} \\
    &= \frac{1}{R} \frac{(1+\kappa R^2)^{-1/2}}{(1-\kappa R^2)^{-3/2}} \left\langle \ss^* g, \sqrt{\ss'}^{-1} I_0 (\Phi^* d_e)^\gamma w^2 \Phi^* f \right\rangle_{ \ss^* (\mu_e^{-2\gamma}) \ss^* (d\Sigma_e^2)} \\
    &\!\!\stackrel{\eqref{eq:murel}}{=} \frac{1}{R} \frac{(1+\kappa R^2)^{-1/2}}{(1-\kappa R^2)^{-3/2}}\left\langle g, (\ss^{-1})^* \sqrt{\ss'}^{-1} I_0 (\Phi^* d_e)^\gamma w^2 \Phi^* f \right\rangle_{ \mu_e^{-2\gamma} d\Sigma_e^2 }, 
\end{align*}
hence the intertwining relation: 
\begin{align*}
    I_0^e d^\gamma = \frac{1}{R} \frac{(1-\kappa R^2)^{3/2}}{(1+\kappa R^2)^{1/2}} (\ss^{-1})^* \sqrt{\ss'}^{-1} I_0 (\Phi^* d_e)^\gamma w^2 \Phi^*,
\end{align*}
which generalizes the specific case $\gamma=0$ in \cite[Eq. (50)]{Monard2019a}. Concatenating the last two calculations, we obtain
\begin{align*}
    (I_0^\sharp \mu^{-1-2\gamma})^e I_0^e d^\gamma &= \frac{1}{R} \frac{(1+\kappa R^2)^{\gamma}}{(1-\kappa R^2)^{\gamma-1}} (\Phi^{-1})^* \frac{1}{w} (I_0^\sharp \mu^{-1-2\gamma}) (\sqrt{\ss'})^{-2\gamma} I_0 (\Phi^* d_e)^\gamma w^2 \Phi^* \\
    &\stackrel{\eqref{eq:murel}}{=} \frac{1-\kappa R^2}{R} (\Phi^{-1})^* \frac{1}{w} I_0^\sharp \mu^{-1} (\ss^*\mu_e)^{-2\gamma} I_0 (\Phi^* d_e)^\gamma w^2 \Phi^*.
\end{align*}
Since the left-hand side is a $C^\infty(\Dm)$-isomorphism, and pullbacks by $\Phi, \Phi^{-1}$ and multiplication by $w, w^{-1}$ are isomorphisms of $C^\infty$ spaces, we deduce that 
\begin{align*}
    (I_0^\sharp \mu^{-1}) (\ss^*\mu_e)^{-2\gamma} I_0 (\Phi^* d_e)^\gamma
\end{align*}
is an isomorphism of $C^\infty(\Dm_R)$. Theorem \ref{thm:CCD} is proved upon setting $d = \Phi^* d_e$ and $t = (\mu (\ss^*\mu_e)^{2\gamma} )^{\frac{1}{2\gamma+1}}$.  

\subsection*{Acknowledgement.}
F.M. and Y.Z. gratefully acknowledge support from NSF CAREER grant DMS-1943580. \F{The authors thank the anonymous referee, whose comments improved the exposition of this article.}

\bibliographystyle{abbrv}

\end{document}